\documentclass[a4paper]{article}

\usepackage{personal_paper}

\usepackage{mathrsfs}

\usepackage{todonotes}

\usepackage{pgfplots}
\pgfplotsset{compat=1.16}

\usepackage[nameinlink,capitalize]{cleveref}

\usepackage[style=alphabetic,backend=bibtex,maxbibnames=99]{biblatex}
\graphicspath{{./graphics/}}

 \makeatletter
  \renewcommand{\ALG@name}{Box}
  \makeatother

\newcommand{\expm}{\operatorname{e}}
\newcommand{\sigX}{\mathbb{X}}

\newcommand{\sigx}{\mathbbm{x}}

\newcommand{\eig}{\operatorname{eig}}
\newcommand{\trac}{\operatorname{tr}}

\providecommand{\martin}[1]{{\textcolor{black}{#1}}}
\providecommand{\martinr}[1]{{\textcolor{black}{#1}}}
\providecommand{\christian}[1]{{\leavevmode\color{black}#1}}

\newcommand{\blue}[1]{{\leavevmode\color{black}#1}}

\newenvironment{customlegend}[1][]{%
	\begingroup
	\csname pgfplots@init@cleared@structures\endcsname
	\pgfplotsset{#1}%
}{%
	\csname pgfplots@createlegend\endcsname
	\endgroup
}%

\def\addlegendimage{\csname pgfplots@addlegendimage\endcsname}

\begin{document}

\title{Dimension reduction for path signatures}

\author{Christian Bayer and Martin Redmann}

\maketitle

\begin{abstract}
This paper focuses on the mathematical framework for reducing the complexity of models using path signatures. The structure of these signatures, which can be interpreted as  collections of iterated integrals along paths, is discussed and their applications in areas such as stochastic differential equations (SDEs) and financial modeling are pointed out. In particular, exploiting the rough paths view, solutions of SDEs continuously depend on the lift of the driver. Such continuous mappings can be approximated using (truncated) signatures, which are solutions of high-dimensional linear systems. In order to lower the complexity of these models, this paper presents methods for reducing the order of high-dimensional truncated signature models while retaining essential characteristics. The derivation of reduced models and the universal approximation property of (truncated) signatures are treated in detail. Numerical examples, including applications to the (rough) Bergomi model in financial markets, illustrate the proposed reduction techniques and highlight their effectiveness.
\end{abstract}



\section{Introduction}
\label{sec:introduction}

The path signature (Chen \cite{chen1957integration}) -- i.e., the sequence of all iterated integrals of the components of a path $x: [0,T] \to \R^d$ -- plays a fundamental role in the analysis of (rough) paths and functions on path space.
On the theoretical side, the signature is the fundamental building block of Lyons' theory of \emph{rough paths} \cite{lyons1998differential}, providing pathwise analysis of differential equations driven by rough paths (i.e., H\"{o}lder continuous paths of arbitrary H\"{o}lder coefficient), see \cite{friz2020course,friz2010multidimensional} for modern accounts of the theory.
Indeed, a rough path is nothing but a path together with a (properly truncated) signature.

Some of the most important properties of path signatures include:
\begin{itemize}
    \item The signature encodes the path up to so-called tree-like excursions, and it completely determines the path, provided that we add running time as a component.
    \item The signature is invariant under time re-parameterization.
    \item The linear functionals of the signature form an algebra of functionals on path space.
\end{itemize}

In applications, signatures play a fundamental role for approximations of functionals on \emph{(rough) path space}, comparable to polynomials on finite-dimensional spaces.
Indeed, it is well-known that linear functionals on signatures are \emph{universal approximators}, i.e., that continuous functions mapping path-space to the real numbers can be approximated by linear functionals of the signature when restricted to compact sets.
Several alternative formulations also exist, we refer to \cite{lee2023signature} for a recent review article and \cite{bayer2023primal} for a version on approximation in an $L^p$-sense.

Of course, similar approximation results are also available in a probabilistic rather than pathwise setting, i.e., when considering stochastic processes rather than single paths.
A corresponding universal approximation result in terms of signatures (of semimartingales) can be found in \cite{cuchiero2022universal}, but we  also like to mention the classical stochastic Taylor expansion in terms of iterated integrals of the Brownian motion.

The path signature has been found to be a powerful feature map in machine learning, when dealing with time series data, see T.~Lyons' ICM lecture \cite{lyons2014rough}, for instance, \cite{perez2018signature,buehler2020generating,morrill2020generalised,lyons2022signature,arrubarrena2024novelty} for specific examples.
Here, both the fact that signatures provide (efficient) encodings of paths (of arbitrary length) in a fixed vector space and the invariance under time re-parameterizations are crucial.

Beyond machine learning, signatures can be used as a fundamental building block for modelling and stochastic analysis of processes with memory -- or, more generally, when we only want to impose minimal assumptions on the model.
An example for the use of signatures in modelling is \cite{cuchiero2022signature}, where a diffusion-type model with coefficients given as linear functionals of the signature of the underlying Brownian motion is fitted to market prices of options.

Beyond modelling, the signature can be a powerful tool for analysis and numerics of non-Markovian processes.
In particular, for BSDEs (see, for instance, \cite{BFZ24}) or stochastic optimal control (\cite{KLPA20,bayer_hager_riedel_book24}) we are often dealing with computations of conditional expectations w.r.t.~the relevant sigma-algebra $\mathcal{F}_t$.
In the Markovian case, this corresponds to a finite-dimensional regression problem.
But in the non-Markovian case (i.e., with memory), such conditionals expectations generally depend on the whole history of the driving path up to time $t$.
Such path-dependent conditional expectations can be efficiently computed as (linear or non-linear) functions of the signature of the underlying path, see \cite{bayer2023primal}.

From a computational point of view, the signature poses severe challenges mainly due to its size.
Indeed, a priori, the signature -- the collection of \emph{all} iterated integrals -- is an infinite-dimensional object, and, hence, not directly computable.
For some applications, the \emph{signature kernel} (see \cite{lee2023signature}) can be used instead, which is actually computable as a solution of a PDE.
Otherwise, the most natural approach is to \emph{truncate} the signature, i.e., to restrict the computations to iterated integrals up to a given degree.
Even so, the size of the (finite) sequence of iterated integrals of degree up to $m$ increases rapidly in terms of $m$ as well as the dimension $d$ of the underlying space: in fact, it is easy to see that there are $1 + d + \cdots + d^m$ such terms -- including one trivial term of order $0$, think of an ``intercept'' in a linear model.




This work addresses these concerns by exploring model order reduction (MOR) techniques specifically tailored for signature-based models. MOR has a long tradition in the field of deterministic control systems \cite{moo1981}, where difficult to control and hard to observe states are neglected to receive a low-dimensional approximation of a large-scale  model. This work laid the foundation for establishing balancing-related MOR for controlled SDEs \cite{beckerhartmann,bennerdamm,redmannbenner, bennerdammcruz}. Complexity reduction in such probabilistic frameworks has an enormous impact, e.g., when sampling methods are applied. \christian{More specific to financial models, an important potential use case of MOR is derivative pricing, requiring many simulations based on one (calibrated) asset price model. For details on the application of MOR to derivative pricing, including in the case of American options, see \cite{mor_heston}.}

However, the above mentioned works rely on certain stability assumptions that are not necessarily satisfied in practice. A way to overcome this issue is to use the ansatz of \cite{morGawJ90} established for controlled systems of ordinary differential equations and extended to SDEs in \cite{redmann_jamshidi}. The control theory perspective on MOR does no longer work when dimension reduction is applied to financial models. Let us refer to \cite{irka_finance, mor_heston}, where a bridge was built between the previously mentioned concepts and complexity reduction for large-scale asset price models. Detached from control concepts and stability assumptions, MOR for signatures is studied in this paper. We develop a robust mathematical framework to reduce the complexity of these models while preserving their essential features, provide a comprehensive method for deriving reduced order models of truncated signatures, with discussions on their theoretical foundations and practical implications. The efficiency of the proposed techniques is illustrated through numerical examples, including an application to the (rough) Bergomi model--a widely used model in financial markets. The results demonstrate the effectiveness of MOR in reducing the dimension without a significant loss of accuracy.

Indeed, we apply MOR to realistic signature models of financial markets, as derived from popular equity models (the classical Bergomi model \cite{bergomi2015stochastic} and the rough Bergomi model \cite{bayer2016pricing}, respectively). \blue{Taking the calibrated signature model as ``ground truth'', we then apply MOR to derive a reduced, low dimensional model, such that European option prices in the reduced model are highly accurate approximations of the prices in the high-dimensional signature model.}
Starting from signature models of dimension $n  = 1365$ in the Bergomi case, we find that a reduced model of dimension $\widetilde{n} = 5$ achieves a relative error of order $10^{-2}$ in the implied volatility of option prices,  $\widetilde{n} = 11$ is sufficient for a relative error of $10^{-4}$, and $\widetilde{n} = 27$ is exact up to machine precision, see Figure~\ref{ImVol_12month_rel_error_r5_r11_r27}. We observe that the same order of accuracy is obtain in the $L^2$-sense, see Figure~\ref{plot_L2error}. This means that our reduced system yields a strong approximation.

In case of the rough Bergomi model, we choose an approximate signature model of dimension $n = 3280$. 
A reduced dimension of $\widetilde{n} = 15$ achieves a relative error of around $10^{-3}$. Furthermore, fixing $\widetilde{n} = 55$ ensures machine precision in the rough case, see Figure~\ref{ImVol_12month_rel_error_r15_r55}. Once more, the $L^2$-error of the approximation is very similar, see Figure~\ref{plot_L2error_rough}.

These numerical illustrations show that MOR can alleviate one of the main drawbacks of signature based models and methods, namely the fast explosion of the signature's dimension in terms of the dimension of the underlying process as well as the truncation level.

\paragraph*{Outline} We start by providing an introduction to path signatures in Section~\ref{sec:signatures}. The basic structures are introduced in the setting of smooth paths and later extended to signatures of continuous semimartingales. We also provide the fundamental approximation results for signatures and minimal necessary results from rough path theory. As a motivation, we summarize results on signature models in finance by \cite{cuchiero2022signature}.

Section~\ref{sec:dimension-reduction} provides a derivation of MOR for signatures. Note that (truncated) signatures of continuous semimartingales solve a linear system of SDEs with specific structure, to which general MOR for SDEs is adapted. The reduced signature equation is introduced in Section~\ref{sec:reduced-order-signature}.

Finally, we provide numerical examples showcasing the success of MOR for signature models in finance in Section~\ref{sec:numerical-examples}.


\section{Signatures}
\label{sec:signatures}


\christian{Signatures can be seen as part of the repertoire of mathematical finance by now. For this reason, we only give a very cursory introduction here, mainly to fix notations.
For readers not yet familiar with signatures, a detailed exposition is given in Appendix~\ref{sec:intr-sign}.}

\christian{Given a finite-dimensional real vector space $V$ and $T>0$, } let $x: [0,T] \to V$ be an $\alpha$-Hölder continuous path ($\alpha \in (0,1]$), and denote by $\mathbf{x} \in \mathscr{C}^\alpha_g([0,T]; V)$ a geometric rough path lift of $x$.
We will often assume that the first component $x^1_t \equiv t$, and denote the set of geometric $\alpha$-Hölder rough paths satisfying this condition by $\widehat{\mathscr{C}}^\alpha_g([0,T]; V)$.

The (truncated) \emph{signature} of $\mathbf{x}$ is the collection of all iterated integrals \christian{of the form
\begin{equation*}
  \int_{s < t_1 < \cdots < t_n < t} \dd x^{i_1}_{t_1} \cdots \dd x^{i_n}_{t_n}, \ i_1, \ldots, i_n \in \{1, \ldots, \dim V\}, \ n \in \mathbb{N}, \ 0 \le s \le t \le T,
\end{equation*}
} (up to order $m \in \N$) and is denoted by $\sigx^{<\infty}_{s,t}$ ($\sigx^{\le m}_{s,t}$, respectively) for $0 \le s \le t \le T$.
It takes values in the extended tensor algebra $T((V)) \coloneqq \prod_{n=0}^\infty V^{\otimes n}$ or its truncated version $T^m(V) \coloneqq \prod_{n=0}^m V^{\otimes n}$.

\christian{In many respects, signatures, or, more precisely, linear functionals on the signature tensor, can be seen as natural generalizations of polynomials from Euclidean space to path-space. In particular, the signature $\sigx^{<\infty}_{0,T}$ (essentially) characterizes the underlying path $x$, and the space of linear functionals on the signature form an \emph{algebra} of functions on path space. As a consequence, the linear functionals on the signature are dense in the continuous functions on path-space on compacts -- see Theorem~\ref{thr:universality} for a precise statement. Hence, signatures provide universal features on path space, and can be seen as natural building blocks for approximation of continuous functions on path-space, similar to the role of polynomials in Euclidean space.}

A particularly important property of signatures from the point of view of MOR is that it is given as the solution of a \emph{linear system of (controlled) differential equations}. Indeed, we formally have
\begin{equation}
  \label{eq:signature-ODE}
  \dd \sigx_{0,t}^{<\infty} = \sigx_{0,t}^{<\infty} \otimes \dd x_t = \sum_{i=1}^d \sigx_{0,t}^{<\infty} \otimes e_i \dd x^i_t, \quad \sigx_{0,0}^{<\infty} = 1 \in T((V)),
\end{equation}
which can be made precise by imposing a suitable topology.
In the case of the truncated signature, the corresponding differential equation
\begin{equation}
  \label{eq:signature-ODE-truncated}
  \dd \sigx_{0,t}^{\le m} = \sigx_{0,t}^{\le m} \otimes \dd x_t = \sum_{i=1}^d \sigx_{0,t}^{\le m} \otimes e_i \dd x^i_t, \quad \sigx_{0,0}^{\le m} = 1 \in T^m(V),
\end{equation}
is given on the finite-dimensional vector space $T^m(V)$, and a proper solution theory is available by, e.g., considering $T^m(V)$ as Euclidean space. In both cases, the differential equations should be considered as \emph{rough differential equations}.

In what follows, we will consider signatures of random paths. More precisely, let $X$ be a \emph{continuous semimartingale}. Analogously, we define the full and truncated signature $\sigX^{<\infty}$ and $\sigX^{\le m}$ in terms of the \emph{Stratonovich integral}. We do allow $X$ to have deterministic components, such as $t$ or $\langle X \rangle_t$.


In the following, we will concentrate on the semimartingale setting, but note that many discussions are equally valid in the (deterministic or stochastic) rough path framework (or even the bounded variation setting).
As seen in~\eqref{eq:signature-ODE-truncated} above, the truncated signature of the semimartingale $X$ satisfies
\begin{equation}\label{eq:signature-SDE-truncated}
  \dd \sigX_{0,t}^{\le m} = \sigX_{0,t}^{\le m} \otimes \circ \dd X_t = \sum_{i=1}^d \sigX_{0,t}^{\le m} \otimes e_i \circ \dd X^i_t, \quad \sigX_{0,0}^{\le m} = 1 \in T^m(V),
\end{equation}
where the tensor product $\otimes$ is truncated to $T^m(V)$.
Since~\eqref{eq:signature-SDE-truncated} is a linear SDE (driven by a semimartingale), it fits within the context of stochastic model order reduction, see, for instance, \cite{irka_finance, mor_heston}. For concreteness, we will rewrite the SDE in terms of coordinates.
We will use the canonical coordinates
\begin{equation*}
  \underbrace{1}_{\in V^{\otimes 0} \subset T^m(V)}, \, \underbrace{e_1, \ldots, e_d}_{\in V \subset T^m(V)}, \, \underbrace{e_1 \otimes e_1 , \ldots e_d \otimes e_d}_{\in V^{\otimes 2} \subset T^m(V)}, \ldots,\, \underbrace{e_1^{\otimes m}, \ldots, e_d^{\otimes m}}_{\in V^{\otimes m} \subset T^m(V)},
\end{equation*}
which we order by length first and lexicographically within elements of the same length.
The linear vector fields $V_i: T^m(V) \to T^m(V)$, $\mathfrak{a} \mapsto \mathfrak{a} \otimes e_i$ (in the sense of the truncated tensor product), can then be represented by matrices $N_i \in \R^{n \times n}$, where
\begin{equation*}
  n = n(d,m) = \sum_{i=0}^m d^i = \frac{d^{m+1}-1}{d - 1} = \dim T^m(V).
\end{equation*}
The matrices $N_1, \ldots, N_d$ can be given in closed form, see Box~\ref{alg:matrix-signature-SDE}.
\begin{algorithm}[!tb]
\caption{Matrices representing the vector fields driving the truncated signature.} \label{alg:matrix-signature-SDE}
\begin{algorithmic}[1]
\State Initialize matrices $N_1= \dots = N_{d}=0$
\For{$i=1, \dots, d$}
\For{$k=1, 2, \dots, (d^m-1)/(d-1)$}
\State $N_i[1+(k-1)d+i,k]=1$
  \EndFor
  \EndFor
\end{algorithmic}
\end{algorithm}
\martin{Usually, $n$ is very large leading to a high-dimensional linear SDE of the form}
\begin{equation}
  \label{eq:signature-SDE-matrices}
  \dd \sigX^{\le m}_{0,t} = \sum_{i=1}^d N_i \sigX^{\le m}_{0,t} \circ \dd X^i_t, \quad \sigX^{\le m}_{0,0} = (1, 0, \ldots, 0)^\top \in \R^n,
\end{equation}
where we use the same symbol $\sigX^{\le m}_{0,t} \in \R^n$ for the version of the signature in coordinates.
\martin{It is important to notice that $N_1, \ldots, N_d$ do not depend on the underlying semimartingale, but only on the dimension $d$ and the truncation level $m$.}
\begin{remark}
  \label{rem:nilpotency-matrices}
  Note that the matrices $N_1, \ldots, N_d$ are \emph{nilpotent} with order $m+1$, i.e., for any $i_1, \ldots, i_m, i_{m+1} \in \set{1, \ldots, d}$ and $1 \le k \le m$ we have
  \begin{equation*}
    N_{i_1} \cdot N_{i_2} \cdots N_{i_k} \neq 0, \text{ but } N_{i_1} \cdot N_{i_2} \cdots N_{i_m} \cdot N_{i_{m+1}}=0. 
  \end{equation*}
  This is clear from the construction: applying the vector field $V_i$ to an element $\mathfrak{a} \in V^{\otimes n} \subset T^m(V)$ yields an element $V_i(\mathfrak{a}) \in V^{\otimes (n+1)} \subset T^m(V)$, for any $0 \le n < m$.
  Hence, iteratively applying the vector fields $m+1$ times always gives the zero element, as we are ``pushed out'' of the truncated tensor algebra.
\end{remark}
\martin{System \eqref{eq:signature-SDE-matrices} is a large-scale and hence it is computationally expensive to evaluate it. Therefore, we propose an MOR scheme in the following section aiming to reduce computational complexity.}

\section{Dimension reduction}
\label{sec:dimension-reduction}

\martin{We propose a projection-based dimension reduction procedure that can, in principle, be applied directly to the Stratonovich equation \eqref{eq:signature-SDE-matrices}. However, we rewrite \eqref{eq:signature-SDE-matrices} as an It\^o stochastic differential equation for two reasons. First, the It\^o case is studied and hence understood much better in the MOR context. Second, there is theoretical evidence that the dimension reduction technique used in this paper performs better when applying it in the It\^{o} framework \cite{mor_fbm}.}
In order to simplify the dimension reduction, we will impose assumptions assuring that the It\^{o} version is still linear.

\begin{assumption}
  \label{ass:brownian-signature}
  $X_t$ is a time-extended Brownian motion, i.e., $X^1_t = t$ and $B_t \coloneqq (B^2_t, \ldots, B^d_t)^\top \coloneqq (X^2_t, \ldots, X^d_t)^\top$ is a $(d-1)$-dimensional Brownian motion with associated matrix $K=(k_{ij})_{i, j=2, \dots, d}$ that determines the covariance by $\mathbb E[B_t B_t^\top] = K t$.
\end{assumption}

\begin{remark}\label{ass:semimartingale}
Alternatively, we can also consider a general continuous semimartingale, provided that the quadratic variation processes are also part of $X$.  The driving process $X$ is a continuous semimartingale such that all quadratic covariations $[X^i,X^j]_t$ are already (constant multiples) of coordinates of $X$.
  More specifically, $X_t = (\widehat{X}_t, \bar{X}_t)$, where $\bar{X}$ is a (proper) continuous semimartingale and $\widehat{X}$ is a bounded variation process such that for any $i,j$ there is a constant $k_{ij}$ and an index $c_{ij}$ such that $[\bar{X}^i, \bar{X}^j] \equiv k_{ij} \widehat{X}^{c_{ij}}$.
Note that Assumption~\ref{ass:brownian-signature} is a special case of this semimartingale setting. For simplicity, we will concentrate on the former case in what follows and leave the situation of this remark for future work.
\end{remark}
We aim to rewrite \eqref{eq:signature-SDE-matrices} in the It\^{o} sense.
The associated It\^{o}-Stratonovich correction term is $0.5 \sum_{i, j=2}^{d} N_i N_j k_{ij}$, where $k_{ij}$ is the $ij$-th entry of the matrix $K$ determining the covariance of the Brownian motion. The drift coefficient is $A= N_1 + 0.5 \sum_{i, j=2}^{d} N_i N_j k_{ij}$ leading to
\begin{align}\label{ito_sig}
 \dd\mathcal X_{t} = A \mathcal X_{t} \dd t+ \sum_{i=2}^{d}N_i \mathcal X_{t} \dd B_{t}^i,   \quad \mathcal X_{0}=z, \quad t\in [0,T],
\end{align}
where $z\in \mathbb{R}^n$ is a generic initial state. \martin{In particular, we are interested in $z=\big(\begin{matrix} 1 & 0 &\dots & 0 \end{matrix}\big)^\top$ giving us the signature process $\mathcal X_t = \sigX^{\le m}_{0,t}$, but we discuss the dimension reduction procedure in more generality. In principle, it is also possible to consider the same MOR scheme yielding an accurate estimate for an entire subspace of initial states but this is beyond the focus of this paper. Moreover, we do not aim to approximate the full state variable $\mathcal X_{t}$ when applying a dimension reduction technique. We rather find a reduced model that is able to reproduce a \martinr{given} linear map of $\mathcal X_t$ represented by a fixed matrix \martinr{$L\in \mathbb R^{p\times n}$}, where $p$ is potentially much smaller than $n$. \martinr{This means that the reduced system is not designed to reproduce $\mathcal X_{t}$, but \begin{align}\label{output_rep}
\mathcal Y_t = L \mathcal X_t.                                                                                                                                                                                                                                                                                                                                                                                                                                                                                                                                                                                                                                                                                                                                                                                                                                                                                                                                                                                                                                                                                                                                                                                                                    \end{align}
The given matrix $L$ depends on the particular application. In the signature context, $L$ is usually learned from data of the nonlinear dynamics that we aim to approximate by an linear map of the signature.} This framework is motivated by the universal approximation property of the signature that tells us that nonlinear time dependent dynamics like rough or stochastic differential equations can be approximated by a linear map of the signature, see Sections \ref{sec:universal_app_sig} and particularly \ref{sec:dynamic-signature}.}
\martin{
\begin{remark}
Considering the general continuous semimartingale setting, using the notation of Remark \ref{ass:semimartingale} and assuming that $\widehat{X}$ takes values in $\mathbb R^{\hat d}$ the It\^{o} SDE in \eqref{ito_sig} is replaced by \begin{align*}
 \dd\mathcal X_{t} = \sum_{k=1}^{\hat d} N_k \mathcal X_{t} \dd \widehat{X}_t^k+0.5 \sum_{i, j=\hat d+1}^{d} N_i N_j \mathcal X_{t} k_{ij} \dd \widehat{X}_t^{c_{ij}}+ \sum_{k=\hat d+1}^{d}N_k \mathcal X_{t} \dd \bar X_{t}^k.                                                                                                                                                                                                                                                                                                                                             \end{align*}
Given that $\bar X$ is a Brownian motion, we have $\hat d=1$ and $\widehat{X}_t=t$ yielding \eqref{ito_sig}.
\end{remark}
}

\subsection{Identifying less relevant signature information}

\martin{The aim of the MOR procedure proposed in this paper is to find a low dimension process $\widetilde {\mathcal X}$ that satisfies a linear SDE and approximates \eqref{ito_sig} in the sense that a linear map of $\widetilde {\mathcal X}$ can accurately reproduce $\mathcal Y$ which usually is a linear map of the signature. Therefore, we derive computable algebraic objects that identify unimportant information in the SDE solved by ${\mathcal X}$ and further detect redundancies in the linear map $L{\mathcal X}$ determined by the matrix $L$.
We introduce the fundamental solution of the signature equation \eqref{ito_sig} in this context. It is} the process $\Phi(t, s)$ satisfying
\begin{align}\label{eq_fund_sol}
  \dd\Phi(t, s) = A \Phi(t, s) \dd t+ \sum_{i=2}^{d}N_i \Phi(t, s) \dd B_{t}^i,   \quad \Phi(s, s)= I\in \mathbb R^{n\times n},
\end{align}
with $s\in [0, T)$. Therefore, we can write $\mathcal X_t = \Phi(t, s) \mathcal X_{s}$ for all $t\in [s, T]$. Based on \martin{the fundamental solution}, we set
\begin{align}\label{P_Gram}
  P&:=\int_0^T \mathbb E\left[\mathcal X_u \mathcal X_u^\top\right] \dd u = \int_0^T \mathbb E\left[\Phi(u, 0)z z^\top \Phi(u, 0)^\top\right] \dd u, \\ \label{Q_Gram}
  Q&:= \int_0^T \mathbb E\left[\Phi(u, 0)^\top L^\top L \Phi(u, 0)\right] \dd u.
\end{align}
\martin{$P$ is defined to detect dominant subspaces in the SDE in \eqref{ito_sig}, whereas $Q$ is needed to identify which state directions are not important in $\mathcal Y= L \mathcal X$. For that reason, $P$ involves the initial state $z$ of the SDE and $Q$ contains the matrix $L$ that represents a linear map of the signature. In this context two representations of $\mathcal X$ are required using different orthonormal bases (ONB) of $\mathbb R^n$.
In particular, we choose ONB consisting of eigenvectors of \martin{$P$ and $Q$} to represent the truncated signature. Let $(p_k)_{k=1, \dots, n}$ be an ONB for $\mathbb R^n$ of eigenvectors of $P$ and $(q_k)_{k=1, \dots, n}$ the one associated with $Q$. \martinr{Since every element of $\mathbb R^n$ can be expanded with respect to any orthonormal basis, we have}
  \begin{align}\label{eigen_rep}
    \mathcal X_t= \sum_{k=1}^n \langle \mathcal X_t, p_{k}
    \rangle_{\mathbb R^n} \,p_k,\quad \mathcal X_t= \sum_{k=1}^n \langle \mathcal X_t, q_{k}
    \rangle_{\mathbb R^n} \,q_k, \quad t\in[0, T].
     \end{align}
\martinr{We focus on these two particular orthonormal bases because they allow us to identify basis vectors that can be neglected in the system dynamics. Specifically, we show below that certain basis vectors $p_k$ and $q_k$ have a negligible contribution in the first and second representation in \eqref{eigen_rep}, respectively.}
First, we find the directions $p_k$, in which $\mathcal X$ is small, since those are not relevant. In order to do so, we look at the associated coefficient  $\langle \mathcal X, p_{k}
    \rangle_{\mathbb R^n}$ in the following proposition.
\begin{proposition}\label{prop_estimates}
  Let $(\lambda_k)_{k=1, \dots, n}$ be the eigenvalues corresponding to the ONB $(p_k)_{k=1, \dots, n}$. Then, we have that
  \begin{align*}
    \mathbb E \int_0^T\langle \mathcal X_u, p_{k}\rangle_{\mathbb R^n}^2\dd u &= \lambda_{k}.
  \end{align*}
\end{proposition}
\begin{proof}
  This identity is trivial as $\mathbb E \int_0^T\langle \mathcal X_u, p_{k}\rangle_{\mathbb R^n}^2\dd u= p_{k}^\top \int_0^T \mathbb E[ \mathcal X_u \mathcal X_u^\top]\dd u\,p_k = p_k^\top P p_k = \lambda_k$ exploiting that $p_k$ has norm $1$.
\end{proof}
Second, we use \eqref{eigen_rep} to answer which directions $q_k$ \martinr{in $\mathcal{X}$ barely contribute to} $\mathcal Y = L \mathcal{X}$. To be more precise, we look at $\mathcal X_{s}$ for each $s\in [0, T)$ and investigate the impact on $\mathcal Y$ on the interval $[s, T]$ when omitting a direction $q_k$. In fact, we study the following difference
\begin{align}\label{y_dif}
 \mathcal Y_t - L\Phi(t, s)\sum_{j=1 \atop j\neq k}^n \langle \mathcal X_s, q_{j}
    \rangle_{\mathbb R^n} \,q_j =  \langle \mathcal X_s, q_{k}
    \rangle_{\mathbb R^n} \,L \Phi(t, s) q_k, \quad t\in[s, T],
\end{align}
for $k\in\{1, \dots, n\}$ using the second representation in \eqref{eigen_rep}  which yields that $\mathcal Y_t =L\Phi(t, s) \mathcal X_{s}= \sum_{j=1}^n \langle \mathcal X_{s}, q_{j}
   \rangle_{\mathbb R^n} \,L\Phi(t, s) q_j$. In the following proposition, the $L^2$-norm of the difference in \eqref{y_dif} is investigated. If it is small, it means that $q_k$ plays a minor role.
\begin{proposition}\label{prop_estimates2}
  Let $(\mu_k)_{k=1, \dots, n}$ be the eigenvalues corresponding to the ONB $(q_k)_{k=1, \dots, n}$. Then, we obtain that
  \begin{align*}
    \mathbb E \int_{s}^T\langle \mathcal X_{s}, q_{k}
    \rangle_{\mathbb R^n}^2\left\|L\Phi(u, s) q_k\right\|_{\mathbb R^p}^2 \dd u &\leq \mu_k  \mathbb E \langle \mathcal X_{s}, q_{k} \rangle_{\mathbb R^n}^2
  \end{align*}
for any $s\in [0, T)$.
\end{proposition}
\begin{proof}
  We note that $\mathcal X_{s}$ is $\mathcal F_{s}$-measurable and that $u\mapsto \Phi(u, s)$ only depends on increments of $B$ after time $s$. Therefore, they are independent and hence we obtain $\mathbb E [\langle \mathcal X_{s}, q_{k}  \rangle_{\mathbb R^n}^2\left\|L\Phi(u, s) q_k\right\|_{\mathbb R^p}^2]= \mathbb E \langle \mathcal X_{s}, q_{k}
  \rangle_{\mathbb R^n}^2\mathbb E\left\|L\Phi(u, s) q_k\right\|_{\mathbb R^p}^2$. Now, we obtain by substitution that
  \begin{align*}
    \mathbb E \int_{s}^T\left\|L\Phi(u, s) q_k\right\|_{\mathbb R^p}^2\dd u &= \mathbb E \int_{s}^T\left\|L\Phi(u-s, 0) q_k\right\|_{\mathbb R^p}^2 \dd u\\
    &=\int_{0}^{T-s}\mathbb E\left\|L\Phi(u, 0) q_k\right\|_{\mathbb R^p}^2\dd u \leq q_k Q q_k= \mu_k.
  \end{align*}
Above, we exploited that $\mathbb E\big[\Phi(u, s) q_k q_k^\top  \Phi(u, s)^\top\big]=\mathbb E\big[\Phi(u- s, 0) q_k q_k^\top \Phi(u-s, 0)^\top\big]$ as both expressions satisfy \eqref{lyap_ODE} with $M=q_k q_k^\top$. This concludes the proof.
\end{proof}
Proposition \ref{prop_estimates} tells us that $p_k$ has a low contribution to $\mathcal X_t$ if the associated eigenvalue $\lambda_k$ is small. Moreover, Proposition \ref{prop_estimates2}  indicates that neglecting $q_k$ in $\mathcal X_{s}$ has a minor impact on the quantity of interest $\mathcal Y$ after time $s$ if $\mu_k$ is small (unless $\mathcal X_{s}$ is very large in the direction $q_k$).} Consequently, we know that we can remove directions $p_k$ and $q_k$ that correspond to small eigenvalues. In order to be able to do this simultaneously, we construct a coordinate transform below that ensures that $p_k=q_k= \expn_k$. Here, $\expn_k$ is $k$-th canonical basis vector of $\mathbb R^n$. Therefore, unimportant directions can be identified with components of a transformed signature. Before stressing this aspect further, a strategy for the computation of $P$ and $Q$ is provided in the following section.

\subsection{Computation of $P$ and $Q$}\label{sec_comp_gram}

\martin{Computability of the reduced system is a crucial aspect of our approach.}
The later dimension reduction procedure relies on having $P$ and $Q$, \martin{defined in \eqref{P_Gram} and \eqref{Q_Gram}}, available. For that reason, let us briefly discuss how these matrices can be computed in practice. We introduce the Lyapunov operator
\begin{align*}
\mathcal L(Z)= A Z+Z A^\top + \sum_{i, j=2}^d N_i Z N_j^\top k_{ij},
\end{align*}
where $Z$ is a matrix of suitable dimension. \martin{$P$ and $Q$ will be computed from equations involving this operator (or its adjoint).  By definition in \eqref{P_Gram}, $P$ is the integral of an expectation. Therefore, we recall the following well-known result that states that this expectation solves a matrix ordinary differential equation.}
\begin{proposition}\label{prop_matrix_ode}
Given $M\in \mathbb R^{n\times n}$, the function $t\mapsto \mathbb E\big[\Phi(t, s) M \Phi(t, s)^\top\big]$ solves \begin{align}\label{lyap_ODE}
 \frac{\dd}{\dd t} Z_t=\mathcal L\left(Z_t\right), \quad Z_{s} =M,\quad t\geq s.                                                                                                                                                                                                                                     \end{align}
\end{proposition}
\begin{proof}
We define $\Phi_t:=\Phi(t, s)$ and make use of It\^{o}'s product rule yielding
\begin{align*} 	                                                                                                                                                                                                                                                                                                                                                                                                                                                                                                                                      
  \dd\big(\Phi_t M \Phi_t^\top\big) &=   \dd\big(\Phi_t\big)M \Phi^\top_t  + \Phi_t M \dd\big(\Phi_t^\top\big) +   \dd\left(\Phi_t \right)M \dd\big(\Phi_t^\top\big)\\
  &=\Big(A \Phi_t \dd t+\sum_{i=2}^d N_i \Phi_t \dd B^i_t\Big) M \Phi_t^\top + \Phi_t M \Big(\Phi_t^\top A^\top \dd t
  +\sum_{i=2}^d \Phi_t^\top N_i^\top \dd B^i_t\Big) \\
  &\quad+  \sum_{i,j=2}^d N_i \Phi_t M \Phi^\top_t N_{j}^\top k_{ij}\dd t.
  \end{align*}
Taking the expectation, we obtain $\frac{\dd}{\dd t}\mathbb E\big[\Phi_t M \Phi_t^\top\big]=\mathcal L\left(\mathbb E\big[\Phi_t M \Phi_t^\top\big]\right)$ \martin{using the Fubini theorem and the fact that the It\^o integral has zero mean. Both properties hold automatically in the Brownian setting. However, in a general semimartingale framework, additional assumptions are required to justify their use.}
\end{proof}
\martin{Let $\vect(M)=\big(\begin{matrix} \mathfrak m_1^\top & \mathfrak m_2^\top &\dots & \mathfrak m_n^\top \end{matrix}\big)^\top$ be the vectorization of a matrix $M$ with $n$ columns $\mathfrak m_1, \dots, \mathfrak m_n$.} Using its relation to the Kronecker product $\tilde\otimes$ between two matrices, \eqref{lyap_ODE} becomes
 \begin{align}\label{vec_lyap_ODE}
 \frac{\dd}{\dd t} \vect(Z_t)=\mathcal K\vect(Z_t), \quad \vect(Z_{s}) =\vect(M),\quad t\geq s,                                                                                                                                                                                                                                     \end{align}
 where the matrix representation of the Lyapunov operator $\mathcal L$ is
\begin{align*}
\mathcal K= A \,\tilde\otimes\, I+I\,\tilde\otimes\, A+\sum_{i, j=2}^d N_i\,\tilde\otimes\,  N_j k_{ij}.
\end{align*}
As $P$ defined in \eqref{P_Gram} is the integral of the solution of \eqref{lyap_ODE} with $s=0$ and $M=z z^\top$, we therefore have that \begin{align*}
 \vect(P) = \int_0^T \vect(Z_t) \dd t = \int_0^T\expm^{\mathcal K t} \vect(z z^\top) \dd t                                                                                                                                                                                                                                                 \end{align*}
inserting the solution representation for \eqref{vec_lyap_ODE}. 
 According to Remark \ref{rem:nilpotency-matrices}, the matrices $N_1, \dots, N_d$ are nilpotent with order $m+1$ leading to $\mathcal K^{j}=0$ for $j\geq 2 m+1$ and hence $\expm^{\mathcal K t} = \sum_{j=0}^{2m} \frac{t^{j}}{j!} \mathcal K^j$.
\martin{\begin{remark}
The degree of nilpotency of $\mathcal K$ can be argued as follows. It is determined by $A \,\tilde\otimes\, I+I\,\tilde\otimes\, A$, since this term involves the non-nilpotent identity matrix. As $A \,\tilde\otimes\, I$ and $I\,\tilde\otimes\, A$ commute, we find that \begin{align*}
\big(A \,\tilde\otimes\, I+I\,\tilde\otimes\, A\big)^n= \sum_{k=0}^n \binom{n}{k}   \big(A \,\tilde\otimes\, I\big)^k \big(I\,\tilde\otimes\, A\big)^{n-k}  =   \sum_{k=0}^n \binom{n}{k} A^k\,\tilde\otimes\, A^{n-k}.                                                                                                                                                                                                                                                         \end{align*}
Using Remark \ref{rem:nilpotency-matrices}, we know that $N_1$ and hence also $A$ is nilpotent of degree $m+1$. Therefore, given that $0\leq n\leq 2m$ there is a $k=0, 1, \dots, n$ such that $A^k\,\tilde\otimes\, A^{n-k} \neq 0$. If $n=2m + 1$, we obtain that $A^k=0$ for $k=m+1, \dots, 2 m+1$ and that $A^{2m+1-k} = 0$ for $k=0, \dots, m$ leading to $\big(A \,\tilde\otimes\, I+I\,\tilde\otimes\, A\big)^{2m+1}=0$.
 \end{remark}
}
Exploiting the nilpotency of $\mathcal K$ now yields
\begin{align}\label{vec_P}
 \vect(P) = \sum_{j=0}^{2m} \frac{T^{j+1}}{(j+1)!} \mathcal K^j\vect(z z^\top).                                                                                                                                                                                                                                                  \end{align}
Although $\mathcal K$ potentially is a huge matrix, $P$ can be computed from \eqref{vec_P} since $\mathcal K$ is extremely sparse making the matrix-vector multiplication cheap. In order to compute $P$ for very large $n$, it is beneficial to devectorize \eqref{vec_P} leading to the explicit representation \begin{align}\label{rep_P}
P = \sum_{j=0}^{2m} \frac{T^{j+1}}{(j+1)!} \mathcal L^j(z z^\top).                                                                                                                                                                                                                                                  \end{align}

The computation of $Q$ can be conducted using similar arguments. Vectorizing \eqref{Q_Gram} yields
  \begin{align*}
 \vect(Q)&= \int_0^T \mathbb E\Big[\vect\Big(\Phi(u, 0)^\top L^\top L \Phi(u, 0)\Big)\Big] \dd u\\
 &= 
 \int_0^T \mathbb E\Big[\Phi(u, 0)^\top \,\tilde\otimes\,  \Phi(u, 0)^\top\Big] \vect\Big(L^\top L\Big) \dd u.
 \end{align*}
Based on Proposition \ref{prop_matrix_ode} and \eqref{vec_lyap_ODE}, we know that \begin{align*}                                                                            \expm^{\mathcal K (t-s)}\vect(M)=\vect\Big(\mathbb E\Big[\Phi(t, s) M \Phi(t, s)^\top\Big]\Big) = \mathbb E\Big[\Phi(t, s) \,\tilde\otimes\,  \Phi(t, s)\Big] \vect(M).                                                                                  \end{align*}
Since this identity is true for all matrices $M$, we find that $ \mathbb E\Big[\Phi(t, s) \,\tilde\otimes\,  \Phi(t, s)\Big] =\expm^{\mathcal K (t-s)}$ and consequently
\begin{align}\label{comp_Q}
 \vect(Q)=
 \int_0^T \expm^{\mathcal K^\top u} \vect\Big(L^\top L\Big) \dd u=\sum_{j=0}^{2m} \frac{T^{j+1}}{(j+1)!} (\mathcal K^\top)^j\vect(L^\top L)
 \end{align}
using once more that $\mathcal K^{j}=0$ for $j\geq 2m+1$. Relation \eqref{comp_Q} now is the basis for the computation of $Q$. Equivalently, we can write \begin{align}\label{rep_Q}
Q
 =\sum_{j=0}^{2m} \frac{T^{j+1}}{(j+1)!} (\mathcal L^*)^j(L^\top L),
 \end{align}
where the Lyapunov operator's adjoint  w.r.t. the Frobenius inner product is \begin{align*}
 \mathcal L^*(Z)= A^\top Z+Z A + \sum_{i, j=2}^d N_i^\top Z N_j k_{ij}.                                                                                           \end{align*}
%

\subsection{Reduced order signature approximation}
\label{sec:reduced-order-signature}

We conduct a state space transformation via a nonsingular matrix $\mathcal T\in \mathbb R^{n\times n}$. To do so, we define a new state variable $\hat {\mathcal X}_t =\mathcal T \mathcal X_t$ and  associated coefficients $(\hat A, \hat N_i, \hat L)=(\mathcal T A \mathcal T^{-1}, \mathcal T N_i\mathcal T^{-1}, L\mathcal T^{-1})$. The purpose of this transformation is that an equivalent system (same quantity of interest $\mathcal Y$) is supposed to be obtained, in which the redundant information can be removed easily by truncation of state components. In particular, $\mathcal T$ is chosen in a way that it simultaneously diagonalizes $P$ and $Q$. Hence, the corresponding eigenvectors $(p_k)$ and $(q_k)$ are the canonical basis $(\expn_k)$ of $\mathbb R^n$ at the same time. Exploiting Propositions \ref{prop_estimates} and \ref{prop_estimates2}, we can then identify unimportant components $\langle \hat {\mathcal X}_{\cdot}, \expn_{k}\rangle_{\mathbb R^n}$ of $\hat {\mathcal X}$ and truncate those. Let us refer to the discussion below these propositions once more and provide further details below.\smallskip

First, we can conclude that the modified signature $\hat {\mathcal X}$ fulfills \begin{equation}
  \label{transformed_sig}
  \dd\hat {\mathcal X}_{t} = \hat A \hat {\mathcal X}_{t} \dd t+ \sum_{i=2}^{d}\hat N_i \hat {\mathcal X}_{t} \dd B_{t}^i,   \quad \hat {\mathcal X}_{0}=\mathcal T z, \quad \mathcal Y_t = \hat L \hat {\mathcal X}_t, \quad t\in [0, T],
\end{equation}
meaning that $\mathcal Y$ is invariant under this transformation. \martinr{Further note that the entries of $\hat{\mathcal X}$ are now linear combinations of iterated integrals given that $\mathcal X$ represents a truncated signature.} We see directly from \eqref{eq_fund_sol} that the fundamental solution of \eqref{transformed_sig} is $\hat \Phi(t, s)= \mathcal T \Phi(t, s)\mathcal T^{-1}$. Consequently, the time-averaged \martin{quadratic forms} of \eqref{transformed_sig} are
\begin{equation}
 \label{transformed_gram}
 \begin{aligned}                                                                     \hat P&:= \int_0^T \mathbb E\left[\hat \Phi(u, 0)\mathcal T z (\mathcal Tz)^\top \hat \Phi(u, 0)^\top\right] \dd u= \mathcal  T P \mathcal T^\top, \\
   \hat Q&:= \int_0^T \mathbb E\left[\hat \Phi(u, 0)^\top \hat L^\top \hat L \hat \Phi(u, 0)\right] \dd u = \mathcal T^{-\top} Q \mathcal T^{-1}.
 \end{aligned}
\end{equation}
The following proposition states the particular transformation required for a simultaneous diagonalization of the time-averaged \martin{quadratic forms}.
\begin{prop}\label{prop_bal}
Given that $P$ and $Q$ are positive definite, we obtain  $\hat P=\hat Q = \Sigma= \diag(\sigma_1,\ldots,\sigma_n)$ using the balancing transformation \begin{equation}\label{bal_transform}
  \mathcal T=\Sigma^{\frac{1}{2}} U^\top L_P^{-1},
\end{equation} 
with the factorization $P=L_PL_P^\top$ and the spectral decomposition $L_P^\top QL_P=U\Sigma^2 U^\top$, where $U$ is orthogonal. Moreover, $\sigma_i^2$ are the eigenvalues of $PQ$.
\end{prop}
\begin{proof}
 Inserting \eqref{bal_transform} in \eqref{transformed_gram}  yields $\hat P= \Sigma^{\frac{1}{2}} U^\top L_P^{-1} P L_P^{-\top} U \Sigma^{\frac{1}{2}}=\Sigma$ and $\hat Q= \Sigma^{-\frac{1}{2}} U^\top L_P^{\top} Q L_P U \Sigma^{-\frac{1}{2}}=\Sigma$. By definition, $\sigma_i^2$ are the eigenvalues of $L_P^\top QL_P$. Exploiting that $L_P^\top QL_P$ has the same spectrum like $L_PL_P^\top Q = PQ$ concludes the proof.
\end{proof}
Given that we have used the balancing transformation \eqref{bal_transform} in \eqref{transformed_sig}, we can identify less relevant directions with components of $\hat {\mathcal X}$ that are associated with small eigenvalues $\sigma_k^2$ of $P Q$. This is a consequence of Propositions \ref{prop_estimates} and \ref{prop_estimates2}, where the transformation ensures that $p_k=q_k=\expn_k$ is the $k$-th column of the $n\times n$ identity matrix and $\lambda_k=\mu_k=\sigma_k$. \martinr{In more detail, Proposition  \ref{prop_estimates} yields \begin{align*}
    \mathbb E \int_0^T\langle \hat {\mathcal X}_u, \expn_{k}\rangle_{\mathbb R^n}^2\dd u &= \sigma_{k}.
  \end{align*}
Hence, if $\sigma_k$ is small, the $k$-th component $\langle \hat{\mathcal X},\expn_k\rangle_{\mathbb R^n}$ of $\hat{\mathcal X}$ has a small contribution in the mean-square sense and can therefore be neglected in the transformed SDE in \eqref{transformed_sig}. Moreover, the corresponding basis vector $\expn_k$ also has only a negligible contribution to the quantity of interest
  $\mathcal Y_t =\hat L\hat \Phi(t, s) \hat{\mathcal X}_{s}= \sum_{j=1}^n \langle \hat{\mathcal X}_{s}, \expn_{j}
   \rangle_{\mathbb R^n} \,\hat L\hat \Phi(t, s) \expn_j$.
Indeed, for any $s\in[0,T)$, Proposition \ref{prop_estimates2} implies
   \begin{align*}
    \mathbb E \int_{s}^T\langle \hat{\mathcal X}_{s}, \expn_{k}
    \rangle_{\mathbb R^n}^2\left\|\hat L\hat \Phi(u, s) \expn_k\right\|_{\mathbb R^p}^2 \dd u &\leq \sigma_k  \mathbb E \langle \hat {\mathcal X}_{s}, \expn_{k} \rangle_{\mathbb R^n}^2.
  \end{align*}
Thus, small values of $\sigma_k$ imply that the corresponding directions contribute only marginally to both the dynamics and the output quantity. }
For that reason, the spectrum of $P Q$ delivers a good truncation criterion and therefore the intuition on how to fix the reduced dimension $\widetilde n\ll n$. In order to find the reduced equation, we partition the solution of \eqref{transformed_sig} with $\mathcal T$ like in  \eqref{bal_transform} as follows $\hat{\mathcal X}_t=\begin{bmatrix}
\hat{\mathcal X}_t^1\\ \hat{\mathcal X}_t^2                                         
\end{bmatrix}$. $\hat{\mathcal X}_t^1$ taking values in $\mathbb R^{\widetilde n}$ corresponds to the large values
  $\sigma_1,\ldots,\sigma_{\widetilde n}$ and $\hat{\mathcal X}_t^2$ to the small values $\sigma_{\widetilde n+1},\ldots,\sigma_n$ in the sense that $\sigma_{\widetilde n+1}\ll \sigma_{\widetilde n}$. The respective partition of the balanced coefficients
  \begin{equation}\label{part_bal}
  \begin{aligned}
\hat A= \begin{bmatrix}{\widetilde A}&\star\\ 
\star&\star\end{bmatrix},\quad \mathcal T z &= \begin{bmatrix}{\widetilde z}\\\star\end{bmatrix},\quad \hat N_{i}= \begin{bmatrix}{\widetilde N}_{i}&\star\\ 
\star&\star\end{bmatrix}, \quad \hat L= \begin{bmatrix}{\widetilde L} &
\star\end{bmatrix}
  \end{aligned}
  \end{equation}
with $\widetilde A, \widetilde N_i\in \mathbb R^{\widetilde n \times \widetilde n}$, $\widetilde z\in \mathbb R^{\widetilde n}$ and $\widetilde L\in \mathbb R^{p \times \widetilde n}$ leads to the reduced system \begin{equation}
  \label{reduced_sig}
  \dd\widetilde{\mathcal X}_{t} = \widetilde A \widetilde {\mathcal X}_{t} \dd t+ \sum_{i=2}^{d}\widetilde N_i \widetilde {\mathcal X}_{t} \dd B_{t}^i,   \quad \widetilde {\mathcal X}_{0}=\widetilde z, \quad \widetilde {\mathcal Y}_t = \widetilde L \widetilde {\mathcal X}_t, \quad t\in [0, T].
\end{equation}
In detail, the reduced model \eqref{reduced_sig} is obtained by removing the equation for $\hat{\mathcal X}_t^2$ in \eqref{transformed_sig} and by setting $\hat{\mathcal X}_t^2=0$ in the dynamics of $\hat{\mathcal X}_t^1$. This is motivated by the minor relevance of $\hat{\mathcal X}_t^2$ leading to a reduced output $\widetilde {\mathcal Y}\approx \mathcal Y$. In particular, choosing $p=1$ and $z=\big(\begin{matrix} 1 & 0 &\dots & 0 \end{matrix}\big)^\top$, the reduced variable $\widetilde {\mathcal X}$ is a candidate for $\widetilde \sigX$ in \eqref{form_aim}.
\martinr{
\begin{remark}
Suppose that the solution of \eqref{ito_sig} represents a truncated signature that is used to learn the dynamics of a nonlinear SDE. If many of the eigenvalues $\sigma_k^2$ of $PQ$ are small, then our results show that the reduced system \eqref{reduced_sig} provides a suitable surrogate for the learning task. Rather than first learning the matrix $L$ in \eqref{output_rep} from data, one may instead learn the dynamics directly by calibrating the coefficients of the reduced system \eqref{reduced_sig}. This replaces the learning of an output map by the identification of the parameters of the reduced-order model.
\end{remark}
}
\subsection{Algebraic error representation}

It is our goal to approximate $\mathcal Y$ in \eqref{output_rep} by the reduced output $\widetilde {\mathcal Y}$ defined in \eqref{reduced_sig}. It is essential to characterize the error of this approach without computing $\mathcal Y$, since this quantity is too computationally involved or not even available. Below, we derive an algebraic representation of the MOR error that can be computed and does not require any sampling of the underlying stochastic processes.
\martinr{
\begin{proposition}
Let $\mathcal Y$ in \eqref{output_rep} be the output of the full model \eqref{ito_sig} and $\widetilde {\mathcal Y}$ be the one of the reduced system \eqref{reduced_sig}. Then, we have
\begin{align}\label{error_rep}
			\int_0^T\mathbb{E}\|\mathcal Y_t-\widetilde {\mathcal Y}_t \|^2_{\mathbb R^p} \dd t = \trac\big( L P\, L^\top\big) + \trac\big(\widetilde L \widetilde P \widetilde L^\top\big)-2 \trac\big(L P_2 \widetilde L^\top\big),
		\end{align}
where $P$ is given by \eqref{P_Gram}, $\widetilde{P}=\int_0^T\widetilde{Z}_t\dd t$ and $P_2=\int_0^T Z_{2, t}\dd t$, where $\widetilde Z$ and $Z_2$ satisfy \begin{align}\label{odebar}
 \frac{\dd}{\dd t} \widetilde Z_t&=\widetilde{\mathcal L}\left(\widetilde Z_t\right), \quad \widetilde Z_{0} =\widetilde z\widetilde z^\top,\quad t\in [0, T],\\ \label{odetilde}
 \frac{\dd}{\dd t} Z_{2, t}&=\mathcal L_2\left(Z_{2, t}\right), \quad Z_{2, 0} =z\widetilde z^\top,\quad t\in [0, T],
\end{align}
with operators $\widetilde{\mathcal L}(\widetilde Z)= \widetilde A \widetilde Z+\widetilde Z \widetilde A^\top + \sum_{i, j=2}^d \widetilde N_i \widetilde Z \widetilde N_j^\top k_{ij}$ and $\mathcal L_2(Z_2)= A Z_2+Z_2 \widetilde A^\top + \sum_{i, j=2}^d N_i Z_2 \widetilde N_j^\top k_{ij}$.
\end{proposition}
\begin{proof}
First, we observe that
\begin{equation}\label{starting_est}
\begin{aligned}
			\int_0^T\mathbb{E}\|\mathcal Y_t-\widetilde {\mathcal Y}_t \|^2_{\mathbb R^p} \dd t&=\mathbb{E}\int_0^T\big\|L\Phi(t,0)z-\widetilde L\widetilde{\Phi}(t,0)\widetilde z\big{\|}_{\mathbb R^p}^2\dd t \\
			&=\mathbb{E}\int_0^T\big\|L^e \Phi^e(t, 0) z^e\big\|_{\mathbb R^p}^2 \dd t\\ &=\trac\big(L^e \int_0^T\mathbb{E}\big[\Phi^e(t, 0) z^e{z^e}^\top\Phi^e(t, 0)^\top\big]\dd t\, {L^e}^\top\big)\\
			&=\trac\big(L^e \int_0^T Z^e_t \dd t\,{L^e}^\top\big),
	\end{aligned}
	\end{equation}
where $Z^e_t = \mathbb{E}\left[\Phi^e(t, 0) z^e{z^e}^\top{\Phi^e}^\top(t, 0)\right]$, $\widetilde{\Phi}$ is the fundamental solution of the reduced model \eqref{reduced_sig} and
\begin{align*}
	z^e=\left(\begin{matrix} z\\ {\widetilde z}\end{matrix}\right),\quad L^e = \left(\begin{matrix} L & {-\widetilde L}\end{matrix}\right),\quad \Phi^e= \left(\begin{matrix}\Phi & 0\\0 &\widetilde{\Phi}\end{matrix}\right).
  \end{align*}
We notice that $\Phi^e$ is the fundamental solution of a system like \eqref{ito_sig} with coefficients
\begin{align*}
	A^e=\left(\begin{matrix} A& 0\\0 &{\widetilde A}\end{matrix}\right), \quad N_i^e=\left(\begin{matrix} N_i& 0\\0 & {\widetilde N}_{i}\end{matrix}\right).
\end{align*}
Therefore, we know by Proposition \ref{prop_matrix_ode} that $Z^e$ satisfies \begin{align}\label{matrixequalforF_error}
\frac{\dd}{\dd t} Z^e_t=\mathcal L^e\left(Z^e_t\right), \quad Z^e_{0} =z^e{z^e}^\top,\quad t\in [0, T],
	\end{align}
with $ \mathcal L^e\left(Z^e\right):=A^eZ^e+Z^e{A^e}^\top+ \sum_{i, j=2}^d N_i^e Z^e{N_j^e}^\top k_{ij}$.
From \eqref{matrixequalforF_error}, it can be seen that the left upper $n\times n$ block of $Z^e$ is the solution $Z$ of \eqref{lyap_ODE} with initial state $zz^\top$. Furthermore,  the right lower  $\widetilde n\times \widetilde n$ block $\widetilde Z$ and the right upper $n\times \widetilde n$ block $Z_2$ of $Z^e$ satisfy \eqref{odebar} and \eqref{odetilde}, respectively.
 From \eqref{starting_est} and the partition $Z^e= \left(\begin{matrix} Z & {Z_2}\\ {Z_2^\top} & {\widetilde Z} \end{matrix}\right)$, we obtain the result of this proposition.
\end{proof}
}
Identity \eqref{error_rep} is now used to calculate the $L^2$-error of the MOR procedure. It is easily accessable, since $P$ is already available as a crucial object in the reduced system computation. Moreover, $\widetilde P$ and $P_2$ are integrals of the solutions of \eqref{odebar} and \eqref{odetilde}. Therefore, they can be computed analogue to $P$, see Section \ref{sec_comp_gram}. On the other hand, $\widetilde P$ and $P_2$ are much easier to obtain as the are low-dimensional objects (depending on $\widetilde n$) in contrast to $P$.


\section{Numerical examples}
\label{sec:numerical-examples}

\subsection{Bergomi model}
\label{sec:bergomi-model}

In the first example, we use a standard stochastic volatility model, namely the \emph{Bergomi model} \cite{bergomi2015stochastic}, as our starting point.
The Bergomi model is highly praised for its flexibility, and the ability to accurately fit equity markets, even with relatively few parameters.
More concretely, the $n$-factor Bergomi model is given by
\begin{subequations}\label{bergomi_model}
\begin{gather}
    \label{eq:bergomi-stock}
    \dd S_t = \sqrt{\xi^t_t} S_t \dd Z_t,\\
    \label{eq:bergomi-forward-variance}
    \dd_t \xi_t^T = \frac{\omega}{\sqrt{\sum_{i,j=1}^n w_i w_j \rho_{ij}}} \xi^T_t \sum_{i=1}^n w_i \mathrm{e}^{-k_i(T-t)} \dd W^i_t,
\end{gather}
\end{subequations}
with initial values $S_0$ for the stock price and an initial forward variance curve $\xi_0^\cdot$.
Here, $Z, W^1, \ldots, W^n$ are correlated standard Brownian motions, where $\rho$ denotes the correlation matrix of $(W^1, \ldots, W^n)$ and -- following \cite{guyon2022vix} -- we denote the correlation between the Brownian motion $Z$ driving the stock price and the Brownian motion $W^i$ by $\rho_{Si}$. \blue{Given that we are mainly interested in derivative pricing for the purposes of the examples, we are always working under a pricing measure and consider forward prices.}

\begin{remark}
    \label{rem:bergomi}
    As indicated above, the Bergomi model is a \emph{forward variance model}, i.e., it models the whole forward variance curve $\xi_t^T \coloneqq \mathbb{E}[v_T \mid \mathcal{F}_t]$, $T \ge t \ge 0$, not just the instantaneous variance $v_t = \xi_t^t$.
    This yields great flexibility for calibration, noting that the initial forward variance curve $\xi_0^\cdot$ -- an essential parameter of the model -- can be read out from market data (variance swaps or vanilla option prices using the \emph{log-strip formula}), and, hence, does not need to be calibrated in principle.
    The specific dynamics of the forward variance in \eqref{eq:bergomi-forward-variance} corresponds to taking the exponential of a weighted sum of Ornstein-Uhlenbeck processes.
\end{remark}

In the first numerical example, we take a two-factor Bergomi model, following the parameterization of \cite[Section 3]{guyon2022vix}, which is a variant of the parameter Set I in \cite[p.~229]{bergomi2015stochastic} with constant initial forward variance curve. Specifically, we choose the parameters presented in Table~\ref{tab:bergomi-parameters}. We here use the notation $\theta_1 = w_1$, and the convention $w_2 = 1 - \theta_1$.

\begin{table}[!htpb]
    \centering
    \begin{tabular}{c|c|c|c|c|c|c|c|c}
        $\omega$ & $k_1$ & $k_2$ & $\theta_1$ & $\rho_{12}$ & $\rho_{S1}$ & $\rho_{S2}$ & $S_0$ & $\xi_0^\cdot$ \\
        \hline
        $3$ & $2.63$ & $0.42$ & $0.69$ & $0.7$ & $-0.9$ & $-0.9$ & $1$ & $0.04$\\
    \end{tabular}
    \caption{Parameters for the Bergomi model used in our numerical example following \cite[Section 3]{guyon2022vix}.}
    \label{tab:bergomi-parameters}
\end{table}

In our first numerical example, we take the Bergomi model as ground-truth\christian{, and fit a signature model to it, which we then reduce following our MOR procedure. In more detail, we proceed according to the following steps.}
\begin{enumerate}
    \item First we train a signature model in the sense of \cite{cuchiero2022signature} with different truncation levels $m$ \christian{minimizing the $L^2$-error between the terminal stock prices $\mathbb{E}\left[ \norm{S_T - S^{\ell}_T}^2 \right]$ for $ S^{(\ell)}_t = \ip{\ell}{\sigX^{\le m}_{0,t}}$, see Section~\ref{sec:exampl-sign-models}. Following the approach of \cite{cuchiero2022signature}, the resulting signature model is free of arbitrage.}
    \item We then vectorize the truncated signature $\sigX^{\le m}_{0,t}$ and formulate the It\^{o}-SDE \eqref{ito_sig} with respective initial state, so that $S^{(\ell)}_t = \mathcal Y_t = L \mathcal X_t$, \christian{for a suitable row-matrix $L$.}
    \item Finally, we approximate the truncated signature by the method described in Section~\ref{sec:reduced-order-signature} and obtain a reduced order signature approximation $S^{(\ell)}_t \approx \widetilde {\mathcal Y}_t= \widetilde{L} \widetilde{\mathcal{X}}_t$ from \eqref{reduced_sig}. \christian{As noted earlier, $\widetilde{\mathcal{X}}_t$ is simply the solution to a (low-dimensional) linear system of SDEs.}
\end{enumerate}
In the specific example of the two-dimensional Bergomi model, we choose the underlying state process $X_t = (t, Z_t, W^1_t, W^2_t)$, i.e., we consider the truncated signature $\sigX^{\le m}_{0,t}$ at level $m$ of a $d$-dimensional process with $d=4$ and $m=5$. This results in an It\^{o}-SDE \eqref{ito_sig} with state dimension $n=1365$ with quantity of interest $\mathcal Y_t$. We apply the dimension reduction scheme of Section \ref{sec:dimension-reduction} to \eqref{ito_sig} in order to obtain a reduced system \eqref{reduced_sig} of order $\widetilde n \ll n$ with quantity of interest $\widetilde {\mathcal Y}_t$. The values $\sigma_k:=\sqrt{\eig_k(PQ)}$ introduced in Proposition  \ref{prop_bal} provide an algebraic criterion for a good choice of $\widetilde n$. According to Section \ref{sec:reduced-order-signature}, the smaller $\sigma_k$, the less important the associated state component in the balanced system \eqref{transformed_sig} is.
\begin{figure}[ht]
 \begin{minipage}{0.45\linewidth}
  \hspace{-0.5cm}
 \includegraphics[width=1.0\textwidth,height=5cm]{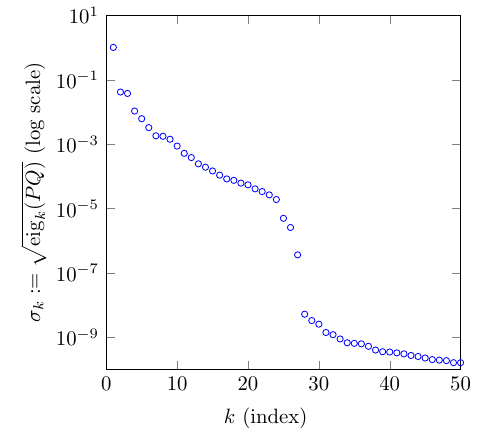}
 \caption{Square root of first $50$ out of $n=1365$ eigenvalues of $PQ$ for signature model associated with \eqref{bergomi_model}.}\label{plot_HSV}
 \end{minipage}\hspace{0.5cm}
 \begin{minipage}{0.45\linewidth}
  \hspace{-0.5cm}
 \includegraphics[width=1.0\textwidth,height=5cm]{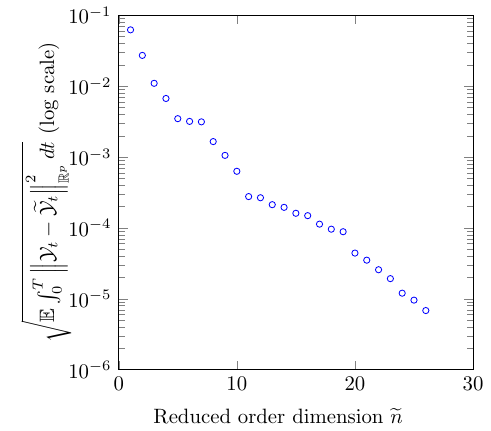}
 \caption{$L^2$-error between output of the signature model of \eqref{bergomi_model} ($n=1365$) and reduced system output for $\widetilde n=1, \dots, 26$.}\label{plot_L2error}
 \end{minipage}
 \end{figure}
Figure \ref{plot_HSV} shows $\sigma_k$ for $k\in \{1, \dots, 50\}$. We observe a strong decay in $k$ and notice that $\sigma_{28}<10^{-8}$ which is below the machine precision. Consequently, \eqref{ito_sig} has a high reduction potential and allows for an exact approximation in case of choosing $\widetilde n=27$. The corresponding reduction errors $\sqrt{\mathbb E \int_{0}^T \left\|\mathcal Y_t - \widetilde {\mathcal Y}_t \right\|_{\mathbb R^p}^2dt}$ can be found in Figure \ref{plot_L2error}. \martin{It is important to notice that this $L^2$-error is not calculated from a Monte-Carlo approach, but from the algebraic error representation in \eqref{error_rep}.} We have a true approximation error for $\widetilde n<27$. If $\widetilde n\geq 27$, the error is numerically zero.
 \begin{figure}[ht]
  \hspace{-0.5cm}
 \includegraphics[width=0.33\textwidth,height=4cm]{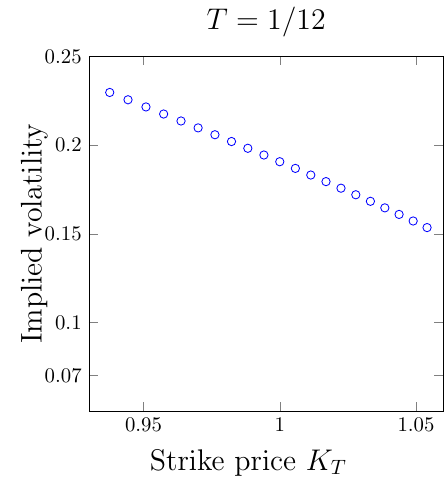}
 \includegraphics[width=0.33\textwidth,height=4cm]{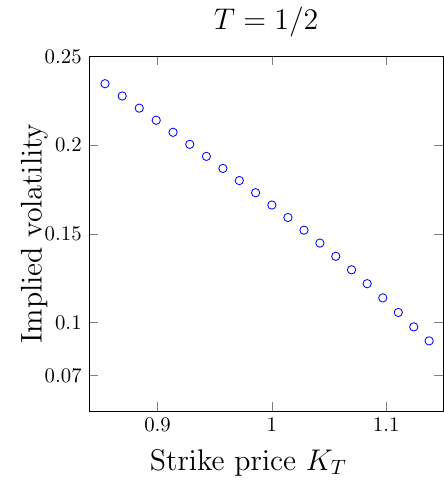}
 \includegraphics[width=0.33\textwidth,height=4cm]{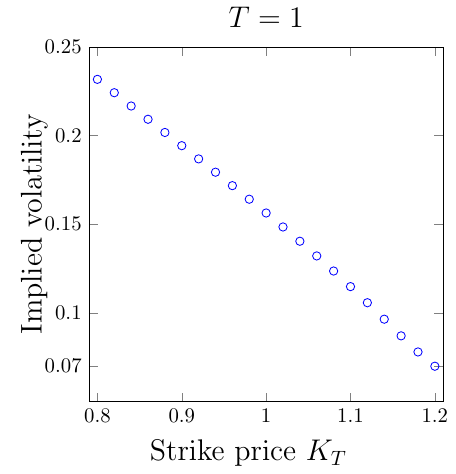}
 \caption{Implied volatilities \martin{of the signature-based approximation $S^{(\ell)}$} of \eqref{bergomi_model} for  $T=1/12, 1/2, 1$ and strike prices $K_T=(0.8+j\cdot 0.02)^{\sqrt{T}}$ with $j=0, 1, \dots, 20$.}\label{imVol_r27_differentT}
 \end{figure}
In addition to the $L^2$-error, the quality of the reduced signature system is tested in a finance context. First, we determine the fair price $\mathbb E\left[ \max\{S^{(\ell)}_T - K_T, 0\}\right]$ of a European option, where the values of the strike price $K_T=(0.8+j\cdot 0.02)^{\sqrt{T}}$ ($j=0, 1, \dots, 20$) are chosen depending on the maturity $T$. The computed prices are then treated as prices from a Black-Scholes model and the associated volatilities are derived (while the interest rate is zero).  These quantities are called implied volatilities (IV). The IV \martin{of $S^{(\ell)}$} are depicted in Figure \ref{imVol_r27_differentT} for $T=1/12, 1/2, 1$.
\begin{figure}[ht]
  \centering
\begin{tikzpicture}
	\begin{customlegend}[legend columns=3, legend style={/tikz/every even column/.append style={column sep=1.0cm}} , legend entries={\,Reduced dim. $\widetilde n=5$
, \,Reduced dim. $\widetilde n=11$ , \,Reduced dim. $\widetilde n=27$}, ]
	\addlegendimage{blue,solid,line width = 1pt, mark options={scale=1.75}, only marks,mark=o}      \addlegendimage{red,line width = 1pt,mark options={scale=1.75}, only marks, mark = x} \addlegendimage{green,line width = 1pt,mark options={scale=1.75}, only marks, mark = star}
	\end{customlegend}
	\end{tikzpicture}
 \includegraphics[width=0.326\textwidth,height=4cm]{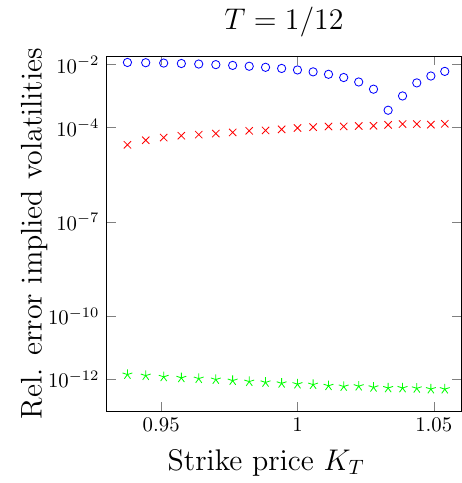}
 \includegraphics[width=0.326\textwidth,height=4cm]{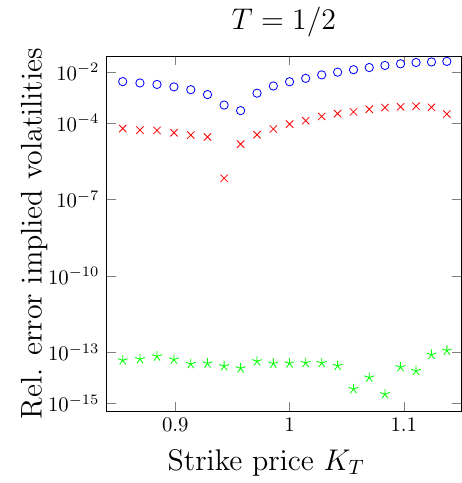}
 \includegraphics[width=0.326\textwidth,height=4cm]{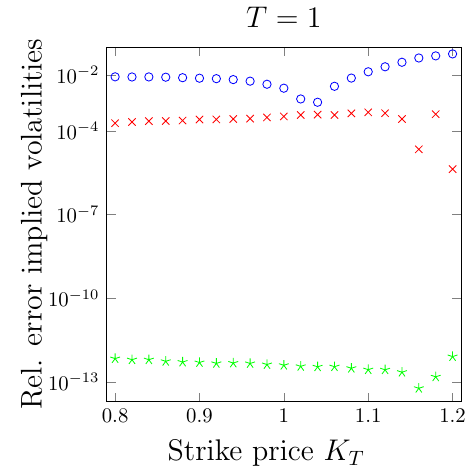}
 \caption{Relative error between volatilities of signature model associated with \eqref{bergomi_model} and reduced systems with $\widetilde n=5, 11, 27$, strike prices $K_T=(0.8+j\cdot 0.02)^{\sqrt{T}}$ and $T=1/12, 1/2, 1$.}\label{ImVol_12month_rel_error_r5_r11_r27}
 \end{figure}
We repeat this procedure in the reduced setting and obtain approximating IV. The relative error between the IV in the full and the reduced dynamics can be found in Figure \ref{ImVol_12month_rel_error_r5_r11_r27} with reduced dimensions $\widetilde n=5, 11, 27$ and terminal times $T=1/12, 1/2, 1$. We can see that the error is around one percent for $T=1/12, 1/2$ if we set $\widetilde n=5$. However, the same reduced dimension shows errors of up to six percent for $T=1$, since relatively small IV come with a higher relative approximation error. Therefore, it can be reasonable to enlarge the reduced dimension to $\widetilde n=11$. This yields relative deviations in the IV of around $10^{-4}$ only and is hence a very good approximation regardless of the maturity $T$. As expected, we obtain IV errors that can be totally neglected when fixing $\widetilde n=27$, see Figure \ref{ImVol_12month_rel_error_r5_r11_r27} once more. This is a very significant reduction in comparison to the original dimension $n=1365$.

 \subsection{Rough Bergomi model}

 Our second example is the \emph{rough Bergomi model} introduced in \cite{bayer2016pricing}, a work-horse model within the class of \emph{rough volatility models}, see \cite{bayer2023rough} for an exposition.
 Intuitively, the rough Bergomi model is a variant of the Bergomi model, where the exponential memory kernel in the variance process is replaced by a fractional kernel, in essence leading to the variance process being an exponential of a fractional Brownian motion.
 While seemingly innocuous, this change destroys the Markovian structure of the resulting model.
 More precisely, we consider the model
\begin{subequations}\label{rough_bergomi_model}
\begin{gather}
    \dd S_t = \sqrt{v_t} S_t \dd Z_t,\\
    v_t = \xi_0(t) \exp\left( \eta  \widehat{W}^H_t - \half \eta^2 t^{2H} \right),
\end{gather}
\end{subequations}
where $ \widehat{W}^H_t \coloneqq \sqrt{2H} \int_0^t (t-s)^{H-1/2} \dd W_s$ denotes a \emph{Riemann--Liouville fractional Brownian motion}, and $W$ and $Z$ are standard Brownian motions with correlation $\rho$.
Note that the rough Bergomi model is -- like the Bergomi model -- a forward variance model, (depending on the initial forward variance curve $\xi_0$).

\begin{table}[!htpb]
    \centering
    \begin{tabular}{c|c|c|c|c}
        $H$ & $\eta$ & $\rho$ & $S_0$ & $\xi_0(\cdot)$ \\
        \hline
        $0.3$ & $2.3$ & $-0.9$ & $1$ & $0.04$ \\
    \end{tabular}
    \caption{Parameters for the rough Bergomi model used in our numerical example.}
    \label{tab:rough-bergomi-parameters}
\end{table}

The term ``rough'' in the \emph{rough Bergomi model} reflects the empirical observation that the Hurst index $H$ should be chosen less than $\half$ -- leading to a power law explosion of the ATM implied volatility skew.
In fact, empirical studies show that $H$ is often chosen very close to $0$, e.g., $H = 0.07$ reported in \cite{bayer2016pricing} based on a calibration on SPX option prices as of February 4, 2010.

We again fit a signature model with underlying state process $X_t = (t, Z_t, W_t)$ (i.e. $d=3$) to a rough Bergomi model, taking model parameters reported in Table~\ref{tab:rough-bergomi-parameters}\christian{, using the same approach as in Section~\ref{sec:bergomi-model} for determining the signature model.}

\blue{
  \begin{remark}
    \label{rem:Lp-approx-signature}
    Note that fitting a signature model to the price $S_T$ in a rough Bergomi model is supported by a corresponding universal approximation. Indeed, as shown in \cite{bayer2023primal} (using \emph{robust signatures}) and \cite{ceylan2025global} (using classical signatures), any $L^p$ r.v.~measurable w.r.t.~a Brownian motion $(Z,W)$ can be approximated by linear functionals of the signature of the time-extended Brownian motion in $L^p$-sense.
  \end{remark}
}

The parameters are realistic, with the possible exception of our choice of $H=0.3$. 
While clearly within the range of Hurst parameters observed in the large scale study \cite{bennedsen2022decoupling}, it is comparatively large for rough volatility models. 
In our experience, fitting a signature model to a rough volatility model with very small $H$ is, however, hard, especially while keeping the truncation degree manageable.
Therefore, we choose $H=0.3$ as a compromise of a clearly fractional, non-Markovian model which is still easily fittable by a signature model.
Note that we again use a constant initial forward variance curve. In particular, we choose $m=7$ as the level of the truncated signature $\sigX^{\le m}_{0,t}$ leading to an equation in \eqref{ito_sig} with state dimension $n=3280$. The reduction technique of Section \ref{sec:dimension-reduction} provides a reduced system \eqref{reduced_sig} with state dimension $\widetilde n \ll n$ with the aim of having an accurate approximation of the quantity of interest, i.e.,  $\widetilde {\mathcal Y}_t \approx {\mathcal Y}_t$. The algebraic values $\sigma_k:=\sqrt{\eig_k(PQ)}$ in Figure \ref{plot_HSV_rough} tell us about the significance of state variables and hence the right reduced dimension $\widetilde n$. We observe that $\sigma_k$ is numerically zero for $k>55$. Therefore, a reduced system of order $\widetilde n=55$ is an exact model. However, choosing $\widetilde n$ much smaller than this can go along with a little approximation error as well. We illustrate the $L^2$-performance for all case with a true error in Figure \ref{plot_L2error_rough}. \martin{Once more, the $L^2$-error is computed from \eqref{error_rep} and not from sampling the underlying processes.} We, e.g., observe that the $L^2$-error between $\widetilde {\mathcal Y}$ with $\widetilde n>10$ and ${\mathcal Y}$ (linear functional of the signature) is below $0.01$.

\begin{figure}[ht]
 \begin{minipage}{0.45\linewidth}
  \hspace{-0.5cm}
 \includegraphics[width=1.0\textwidth,height=5cm]{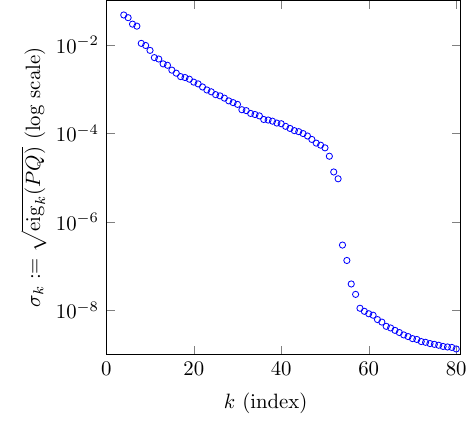}
 \caption{Square root of first $80$ out of $n=3280$ eigenvalues of $PQ$ for signature model associated with \eqref{rough_bergomi_model}.}\label{plot_HSV_rough}
 \end{minipage}\hspace{0.5cm}
 \begin{minipage}{0.45\linewidth}
  \hspace{-0.5cm}
 \includegraphics[width=1.0\textwidth,height=5cm]{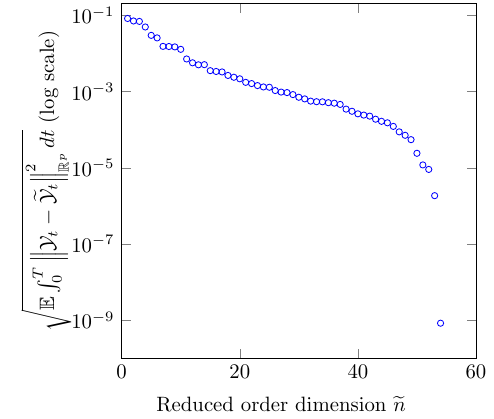}
 \caption{$L^2$-error between output of the signature model of \eqref{rough_bergomi_model} ($n=3280$) and reduced system output for $\widetilde n=1, \dots, 54$.}\label{plot_L2error_rough}
 \end{minipage}
 \end{figure}
 Moreover, we are interested in the quality of the signature approximation \eqref{reduced_sig} when IV are aimed to be reproduced. First of all, let us note the true IV of the signature model in Figure \ref{imVol_r27_differentT_rough} for $T=1/12, 1/2, 1$.
 \begin{figure}[ht]
  \hspace{-0.5cm}
 \includegraphics[width=0.33\textwidth,height=4cm]{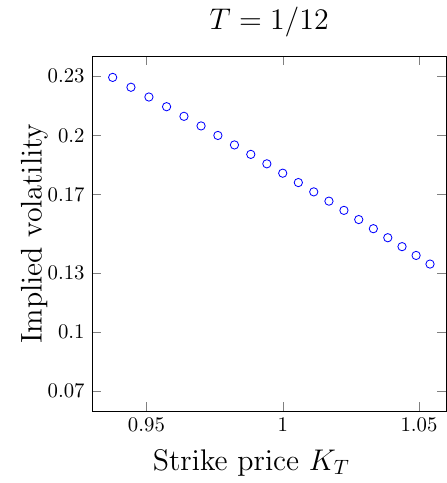}
 \includegraphics[width=0.33\textwidth,height=4cm]{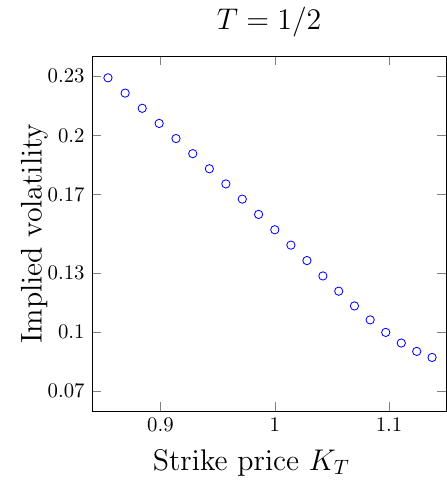}
 \includegraphics[width=0.33\textwidth,height=4cm]{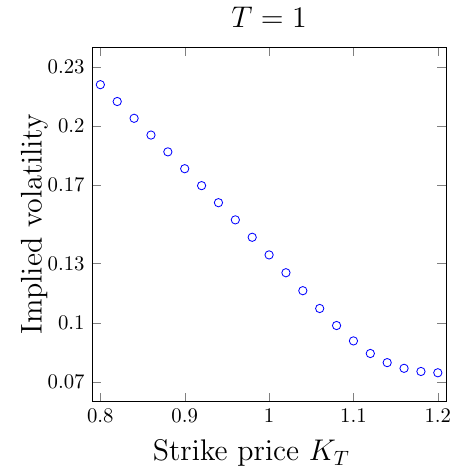}
 \caption{Implied volatilities \martin{of the signature-based approximation $S^{(\ell)}$} of \eqref{rough_bergomi_model} for  $T=1/12, 1/2, 1$ and strike prices $K_T=(0.8+j\cdot 0.02)^{\sqrt{T}}$ with $j=0, 1, \dots, 20$.}\label{imVol_r27_differentT_rough}
 \end{figure}
We approximated these values by the IV of reduced system for $\widetilde n=55$ and, as expected, we obtain an error that can be fully neglected, see Figure \ref{ImVol_12month_rel_error_r15_r55}. We demonstrate the case of $\widetilde n=15$ in the same figure to illustrate that even a reduction to such a small dimension (in comparison to $n=3280$), a relative error of less than $0.01$ can be guaranteed (often around $10^{-3}$).
\begin{figure}[ht]
  \centering
\begin{tikzpicture}
	\begin{customlegend}[legend columns=2, legend style={/tikz/every even column/.append style={column sep=1.0cm}} , legend entries={\,Reduced dimension $\widetilde n=15$ , \,Reduced dimension $\widetilde n=55$}, ]
	     \addlegendimage{red,line width = 1pt,mark options={scale=1.75}, only marks, mark = x} \addlegendimage{green,line width = 1pt,mark options={scale=1.75}, only marks, mark = star}
	\end{customlegend}
	\end{tikzpicture}
 \includegraphics[width=0.326\textwidth,height=4cm]{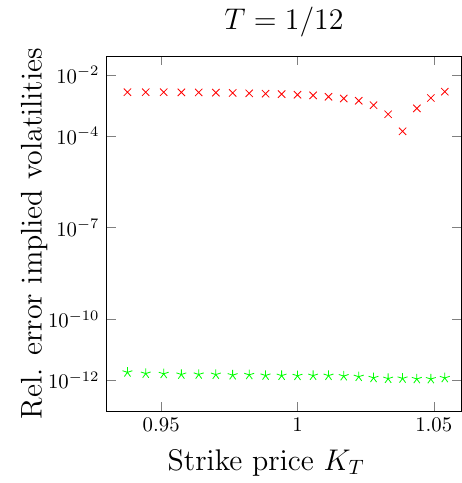}
 \includegraphics[width=0.326\textwidth,height=4cm]{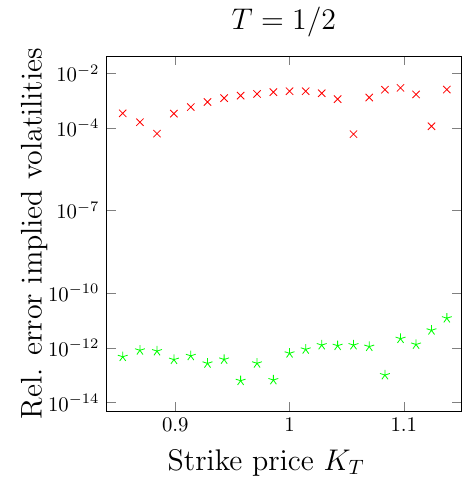}
 \includegraphics[width=0.326\textwidth,height=4cm]{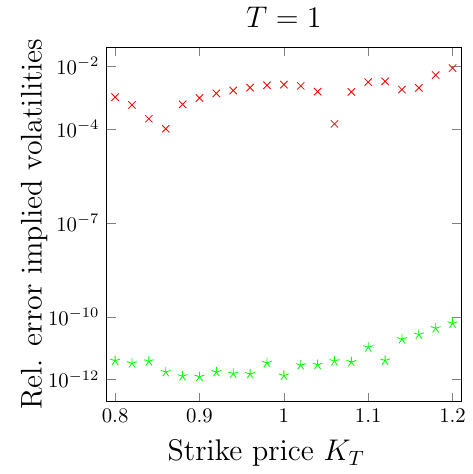}
 \caption{Relative error between volatilities of signature model associated with \eqref{rough_bergomi_model} and reduced systems with $\widetilde n=15, 55$, strike prices $K_T=(0.8+j\cdot 0.02)^{\sqrt{T}}$ and $T=1/12, 1/2, 1$.}\label{ImVol_12month_rel_error_r15_r55}

 \end{figure}


\section*{Acknowledgments}
CB acknowledges support from DFG CRC/TRR 388 ``Rough Analysis, Stochastic Dynamics and Related Fields'', Projects B01 and B03.

 MR is supported by the DFG via the individual grant ``Low-order approximations for large-scale problems arising in the context of high-dimensional PDEs and spatially discretized SPDEs'' -- project number 499366908 and via the SFB 1294, project A07 ``Data-based model order reduction for stochastic dynamics'' -- project number 318763901.

\appendix

\section{Introduction to signatures}
\label{sec:intr-sign}

\christian{For convenience, this section offers a self-contained, conceptual summary of path signatures to support the main results. For more details on the underlying concepts of rough paths, we refer to \cite{friz2020course} for a modern introduction as well as \cite{friz2010multidimensional} for a classical monograph. More information on path signature methods specifically can be found in \cite{sig_methods_finance}.}

\subsection{Signatures of smooth paths}
\label{sec:sig-smooth-path}


Consider a smooth (i.e., $\mathcal C^1$) path\footnote{In fact, all results of this subsection remain valid for continuous bounded variation paths.} $x: [0,T] \to V$ taking values in a Banach space $V$ -- for our purposes, we assume $d \coloneqq \dim V < \infty$.
Given a multi-index $I = (i_1, \ldots, i_n) \in \set{1, \ldots, d}^n$, $n \ge 1$, we denote the iterated integrals
\begin{equation}
  \label{eq:iterated-integrals}
  \sigx^I_{s,t} \coloneqq \int_{s < t_1 < \cdots < t_n < t}  \dot{x}^{i_1}_{t_1} \dd t_1 \cdots  \dot{x}^{i_n}_{t_n} \dd t_n \eqqcolon \int_{s < t_1 < \cdots < t_n < t} \dd x^{i_1}_{t_1} \cdots  \dd x^{i_n}_{t_n} \in \R,
\end{equation}
where $0 \le s \le t \le T$.
Here, $x^i_t$ denotes the $i$th coordinate of the path $x$ evaluated at time $t$ for a fixed basis $e_1, \ldots, e_d$ of $V$. We prefer to use the notation ``$\dd x_t$'' rather than ``$\dot{x}_t \dd t$'' because the former corresponds to the standard notation in the non-smooth case to be considered later -- think of paths of Brownian motion.

Note that the collection of all iterated integrals $\sigx^I_{s,t}$ for $I \in \set{1, \ldots, d}^n$ with fixed $n \ge 1$ takes values in $V^{\otimes n}$ recursively defined by $V^{\otimes 0} \coloneqq \R$, and $V^{\otimes (n+1)} \coloneqq V^{\otimes n} \otimes V$.
Specifically, we write
\begin{equation*}
  \left( \sigx^I_{s,t} \right)_{I \in \set{1, \ldots, d}^n} \eqqcolon \int_{s < t_1 < \cdots < t_n < t} \dd x_{t_1} \otimes \cdots \otimes \dd x_{t_n} \in V^{\otimes n}.
\end{equation*}
The \emph{signature} is the collection of all such iterated integrals, formally 
\begin{equation}
  \label{eq:signature-bv}
  \sigx^{< \infty}_{s,t} \coloneqq 1 + \sum_{n=1}^\infty \int_{s < t_1 < \cdots < t_n < t} \dd x_{t_1} \otimes \cdots \otimes \dd x_{t_n} \in T((V)),
\end{equation}
with the \emph{extended tensor algebra} $T((V))$ being defined as
\begin{equation}
  \label{eq:extended-tensor-algebra}
  T((V)) \coloneqq \prod_{n=0}^\infty V^{\otimes n}.
\end{equation}
Here, the initial term ``$1$'' is considered as the (sole) entry in $V^{\otimes 0} \simeq \R$.
Note that elements of the extended tensor algebra $T((V))$ have infinitely many terms, and are, hence, comparable to formal power series -- in $d$ non-commuting variables $e_1, \ldots, e_d$.
For numerical purposes, signatures need to be truncated at a finite degree. Hence, we also define 
\begin{equation}
  \label{eq:truncated-signature-bv}
  \sigx^{\le m}_{s,t} \coloneqq 1 + \sum_{n=1}^m \int_{s < t_1 < \cdots < t_n < t} \dd x_{t_1} \otimes \cdots \otimes \dd x_{t_n} \in T^m(V) \coloneqq \bigoplus_{n=0}^m V^{\otimes n}.
\end{equation}
Following our analogy from above, an element of the truncated tensor algebra $T^m(V)$ can be compared with a polynomial of degree $m$ -- note, however, that the tensor product is not commutative.

Both $T((V))$ and $T^m(V)$ are algebras under the tensor product $\otimes$ -- the product of formal power series in non-commuting variables. We refer to \cite{bayer2021optimal} for more information.
Using this product, we see that the signature (formally) solves a controlled ODE, namely $\dot{\sigx}_{0,t}^{<\infty} = \sigx_{0,t}^{<\infty} \otimes \dot{x}_t$ or (in a notation easier to adapt for non-smooth paths)
\begin{equation*}
  \dd \sigx_{0,t}^{<\infty} = \sigx_{0,t}^{<\infty} \otimes \dd x_t = \sum_{i=1}^d \sigx_{0,t}^{<\infty} \otimes e_i \dd x^i_t, \quad \sigx_{0,0}^{<\infty} = 1 \in T((V)),
\end{equation*}
where $e_i$, $i=1, \ldots d$, denote the standard basis vectors of $V$ as lifted to elements of $T((V))$, see also~\eqref{eq:signature-ODE}.
The truncated signature satisfies the same ODE, but interpreted on the truncated tensor algebra $T^m(V)$, i.e., with a nilpotent tensor product, which will -- abusing notation -- still be denoted $\otimes$.
That is, the truncated signature satisfies
\begin{equation*}
  \dd \sigx_{0,t}^{\le m} = \sigx_{0,t}^{\le m} \otimes \dd x_t = \sum_{i=1}^d \sigx_{0,t}^{\le m} \otimes e_i \dd x^i_t, \quad \sigx_{0,0}^{\le m} = 1 \in T^m(V).
\end{equation*}

\begin{remark}
  \label{rem:topology-T((V))}
  We do not endow $T((V))$ with a topology, and, hence, the ODE \eqref{eq:signature-ODE} is only defined in a formal way.
  Several topologies for $T((V))$ have been considered in the rough path literature, including natural Hilbert or Banach sub-spaces of $T((V))$, see \cite{chevyrev2018signature} for more details.
  In this paper, we will only ever use the full signature defined on $T((V))$ for motivation, the mathematical analysis will take place on $T^m(V)$ for finite truncation $m$.
  As a finite dimensional vector space, $T^m(V)$ will be equipped with the standard Euclidean metric.
\end{remark}

The signature is invariant under re-parameterization of the path: I.e., if we consider a (smooth, increasing) map, say, $\gamma : [R,S] \to [0,T]$, and a path $\bar{x}: [R,S] \to V$, $u \mapsto x_{\gamma(u)}$, then $\sigx_{\gamma(s),\gamma(t)}^{<\infty} = \bar{\sigx}_{s,t}^{<\infty}$, $R \le s \le t \le S$.
Up to re-parameterization and so-called tree-like excursions, the signature $\sigx^{< \infty}_{0,T}$ uniquely determines the path $x$, see \cite{hambly2010uniqueness}.
Invariance under parameterization as well as the possibility of tree-like excursions can be avoided by adding the component $t$ to the path.
We will often prefer to work with time-extended paths -- i.e., paths with a component $x^1_t \equiv t$.

We will also need to consider the dual algebra $\mathcal{W}_d \simeq T(V^\ast) \coloneqq \bigoplus_{n=0}^\infty (V^\ast)^{\otimes n}$ of linear functionals on $T((V))$.
Here, $\mathcal{W}_d$ denotes the linear span of all words $\mathbf{w} = i_1 \cdots i_k$, $k \ge 0$, in the alphabet $\set{1, \ldots, d}$, which is an algebra under the concatenation product on words, extended with the distributive property.
Consider a generic linear functional $\ell \in \mathcal{W}_d$ applied to a generic element $\mathfrak{a} \in T((V))$.
We can represent $\ell$ as a linear combination $\ell = \sum_{i=1}^k \gamma_i \mathbf{w}_i$ of words $\mathbf{w}_j = i^j_1 \cdots i^j_{K_j}$ of length $K_j$, $j=1, \ldots, k$, for some $k \ge 1$, and
\begin{equation*}
  \mathfrak{a} = \sum_{n=0}^\infty \sum_{I = (i_1, \ldots, i_n)\in \set{1, \ldots, d}^n} \alpha_I e_{i_1} \otimes \cdots \otimes e_{i_n}  \in T((V)),
\end{equation*}
owing to the fact that words $i^j_1 \cdots i^j_{K_j}$ form a basis of $\mathcal{W}_d$ and tensor products $e_{i_1} \otimes \cdots \otimes e_{i_n}$ form a basis of any truncated tensor algebra $T^m(V)$.\footnote{The infinite sum defining $\mathfrak{a}$ needs to be understood as a formal sum.}
We set
\begin{equation}\label{eq:linear-functional}
  \ip{\ell}{\mathfrak{a}} \coloneqq \sum_{j=1}^k \gamma_j \ip{\mathbf{w}_j}{\mathfrak{a}} \coloneqq  \sum_{j=1}^k \gamma_j \alpha_{(i^j_1, \ldots, i^j_{K_j})} \in \R.
\end{equation}

\subsection{Signatures for Hölder continuous paths}
\label{sec:sig-holder-paths}

Let $x \colon [0,T] \to V$ be an $\alpha$-Hölder continuous path, $\alpha \in (0,1]$. This means that the $\alpha$-Hölder seminorm
\begin{align*}
    \sup_{s < t} \frac{\|x_t - x_s\|}{|t-s|^{\alpha}}
\end{align*}
is finite. Then, we write $x \in C^{\alpha}$. We briefly sketch the concepts of rough paths and (truncated) signatures for $x \in C^{\alpha}$ -- see \cite{friz2020course} for more details. Formally, a rough path $\mathbf{x}$ is a two-parameter function from the simplex $\set{0 \le s \le t \le T}$ taking values in the truncated tensor algebra $T^m(V)$. 

\begin{remark}
  \label{rem:rough-path-heuristic}
  The different ``levels'' of a rough path increment $\mathbf{x}_{s,t}$, $s \le t$, have different interpretations. The $V$-valued component is the increment $x_{s,t} = x_t - x_s$ of the underlying $C^\alpha$ path $x$ itself. However, Hölder paths are not regular enough to allow us to solve controlled differential equations of the form 
  \[
    \dd y_t = U(y_t) \dd x_t.
  \]
  However, a formal Taylor expansion shows that higher order Euler approximations of the differential equation in terms of the iterated integral of $x$ of order up to $m = \floor{1/\alpha}$ \emph{would converge}. Of course, the catch is that -- once again -- $C^\alpha$ paths are not regular enough for the iterated integrals to make sense in a classical way when $\alpha \le 1/2$. On the flip side, if we enhance the path increments $x_{s,t}$ with higher order terms $\sigx^{I}_{s,t}$ ``behaving like iterated integrals'' (see \cite{friz2020course} for details), then we can solve the corresponding controlled differential equations for smooth enough vector fields in a pathwise, deterministic way.

  It is well-known that such \emph{rough path lifts} of $C^\alpha$ paths $x$ are \emph{always} possible, but \emph{not} unique when $\alpha \le 1/2$ (think about It\^{o} versus Stratonovich solutions of stochastic differential equations -- corresponding to two different rough path lifts of Brownian motion).
\end{remark}

We introduce a metric for two parameter functions taking values in $T^m(V)$ by 
\begin{align*}
    \varrho_{\alpha}(\mathbf{x}, \widetilde{\mathbf{x}}) := \sum_{n = 1}^m \sup_{0 \le s < t \le T} \frac{\|\sigx^{(n)}_{s,t} - \tilde{\sigx}^{(n)}_{s,t} \|}{|t-s|^{n \alpha}},
\end{align*}
where $\mathbf{x}_{s,t} \coloneqq 1 + \sum_{n=1}^m \sigx^{(n)}_{s,t} \in T^m(V)$ and $\tilde{\mathbf{x}}$ is defined accordingly -- with $\sigx^{(1)}_{s,t} = x_{s,t} \in V$ and $\sigx^{(n)}_{s,t} \in V^{\otimes n}$, $n=2, \ldots, m$, \christian{see \cite[Def.~2.3]{friz2020course} for the case $m=2$ and \cite[Def.~8.6]{friz2010multidimensional} for the general case. It is at the heart of \emph{rough path theory}, and allows the extension of the theory of signatures -- including universal approximation -- from the bounded variation case to the semimartingale case by providing a path-wise integration theory as well as path-wise well--posedness of stochastic differential equations, we refer to the two references above for details.} Now, we can define the desired concepts.
\begin{definition}
  \label{def:alpha-Holder-RP}
  Let $x \in C^{\alpha}$ and choose $N \coloneqq \floor{1/\alpha}$. A two-parameter function 
  \begin{align*}
      \mathbf{x}_{s,t} = 1+\sum_{n=1}^N \sigx^{(n)}_{s,t}\in T^N(V)
  \end{align*}
  with $\sigx^{(1)}_{s,t} = x_{s,t}$ is called a \emph{(geometric $\alpha$-Hölder) rough path associated to $x$} if there exists a sequence of smooth paths $x^{\epsilon}$ with truncated signature  $(\sigx^\epsilon)^{\le N}_{s,t}$, such that 
  \begin{align*}
      \varrho_{\alpha}(\mathbf x, (\sigx^\epsilon)^{\le N}) \to 0, \quad \text{as }\epsilon \to 0.
  \end{align*}
  We denote the set of all geometric $\alpha$-Hölder rough paths by $\mathscr{C}^\alpha_g([0,T];V)$. We also introduce $\widehat{\mathscr{C}}^\alpha_g([0,T]; V) \coloneqq \Set{ \mathbf{x} \in \mathscr{C}^\alpha_g([0,T];V) | \forall t \in [0,T]:\, x^1_{0,t} = t}$, where $x_{s,t}$ denotes the level-$1$ component of $\mathbf{x}_{s,t}$.
\end{definition}
For general $m$, the (truncated) signature $\sigx^{\le m}$ and $\sigx^{<\infty}$ of a geometric $\alpha$-Hölder rough path $\mathbf x$ can be defined as the limit of $(\sigx^\epsilon)^{\le m}$ and $(\sigx^\epsilon)^{<\infty}$, respectively, as $\epsilon \to 0$. \martin{Below, $L(V, W)$ denotes the set of all bounded linear operators from $V$ to the Banach space $W$.}
\begin{definition}
  Suppose that $W$ is another (finite-dimensional) Banach space, $\mathbf x\in\mathscr{C}^\alpha_g([0,T];V)$ and $U:W\to L(V, W)$. An $\alpha$-Hölder path $y:[0,T]\to W$ is called a \emph{solution of the rough differential equation}
  \begin{align}\label{rde}
      \dd y(t) = U(y(t)) \, \dd \mathbf{x}_t, \quad y(0) = y_0\in W,
  \end{align}
  if $y(0) = y_0$ and for a sequence $(x^\epsilon)$ of smooth paths with $\varrho_{\alpha}(\mathbf x, (\sigx^\epsilon)^{\le N}) \to 0$ (as $\epsilon\to 0$), the solutions $y^{\epsilon}$ of
  \begin{align*}
      \dd y^{\epsilon}(t) = U(y^{\epsilon}(t)) \, \dd {x}^{\epsilon}_t, \quad y^{\epsilon}(0) = y_0,
  \end{align*}
  exist and converge in the $\alpha$-Hölder metric to $y$.
\end{definition}
By construction of a solution of a rough differential equation, the truncated signature solves the equation
\begin{equation}
  \label{eq:signature-RDE-truncated}
  \dd \sigx_{0,t}^{\le m} = \sigx_{0,t}^{\le m} \otimes \dd \mathbf x_t = \sum_{i=1}^d \sigx_{0,t}^{\le m} \otimes e_i \dd \mathbf x^i_t, \quad \sigx_{0,0}^{\le m} = 1 \in T^m(V).
\end{equation}
Next, we formulate a result on existence and uniqueness of a solution of \eqref{rde} as well as on properties of the solution map.
\begin{theorem}\label{thm_ex_un_con}
 Given $\mathbf x\in \mathscr C^\alpha([0,T],V)$ for $\alpha\in\left(0, \frac{1}{2}\right]$ and $U\in \mathcal C_b^{N+1}(W, L(V, W))$ or linear, where $N \coloneqq \floor{1/\alpha}$. Then,
there is a unique solution $y\in C^\alpha([0,T], W)$ for \eqref{rde}. Moreover, the solution map $\mathbf x\mapsto g(\mathbf x)=y$ of \eqref{rde} is locally Lipschitz continuous.
\end{theorem}
\begin{proof}
We refer to \cite{friz2020course} for the case $\alpha>1/3$. The more general ($p$-variation) framework can be found in \cite{friz2010multidimensional}.
\end{proof}
\begin{remark}
  \label{rem:RP-principle}
  As evidenced by the difference between It\^{o} and Stratonovich solutions to SDEs, it is, in general, not possible to find solutions to controlled differential equations driven by rough signals $x$ such that the solution map ($x \mapsto y$) is continuous in a pathwise sense. The rough path approach factorizes this map into two parts:
  \begin{enumerate}
  \item the lift of $x$ to a rough path $\mathbf{x}$ (discontinuous);
  \item the solution map $\mathbf{x} \mapsto y$ (locally Lipschitz continuous by Theorem~\ref{thm_ex_un_con}).
  \end{enumerate}
  This fact is sometimes called the ``rough path principle''.
\end{remark}

\subsection{Signatures of semimartingales}
\label{sec:sign-semi-mart}

If, instead of a single deterministic path $x$, we are given a continuous semimartingale $X:\Omega \times [0,T] \to V$, then the statements above remain true, mutatis mutandis, provided that iterated integrals are defined as \emph{Stratonovich integrals} rather than It\^o integrals.

We define the (truncated) signatures $\sigX^{<\infty}_{s,t}$, $\sigX^{\le m}_{s,t}$ as in \eqref{eq:signature-bv} and \eqref{eq:truncated-signature-bv}, but with the iterated integrals \eqref{eq:iterated-integrals} replaced by 
\begin{equation}
  \label{eq:iterated-integrals-stratonovich}
  \sigX^I_{s,t} \coloneqq \int_{s < t_1 < \cdots < t_n < t} \circ \dd X^{i_1}_{t_1} \otimes \cdots \otimes \circ \dd X^{i_n}_{t_n}.
\end{equation}
Indeed, we can see that the thus defined rough path lift $\mathbf{X}$ a.s.~takes values in $\mathscr{C}^\alpha_g([0,T]; V)$ for any $1/3 < \alpha < 1/2$ -- we again refer to \cite{friz2020course} for details. Moreover, the solution of a stochastic differential equation, say
\[
  \dd Y_t = U(Y_t) \circ \dd X_t, \quad Y_0 = y_0 \in W,
\]
coincides a.s.~with the (pathwise) solution of the rough differential equation
\[
  \dd Y_t = U(Y_t) \dd \mathbf{X}_t, \quad Y_0 = y_0 \in W,
\]
where $U: W \to L(V,W)$ is sufficiently smooth.

\begin{remark}
  \label{rem:quadratic-variation}
  From a modeling point, it may be advantageous to include the quadratic variation of $X$ in the construction of the signature, i.e., to consider the signature of $(t, X_t, \langle X \rangle_t)$.
  While it is included in our framework, we will not require such an extension.
\end{remark}


\begin{remark}
  \label{rem:jump-semimartingales}
  The setting can also incorporate semimartingales with jumps, see, for instance, \cite{cuchiero2022universal} for a precise statement and proof of universality of signatures in that case.
\end{remark}


\subsection{Universal approximation theorem}
\label{sec:universal_app_sig}

We will now formulate \emph{universality} of signatures, a well known result in the literature (see, for instance, \cite{KLPA20,bayer2023primal}).
Recall that $\widehat{\mathscr{C}}^\alpha_g([0,T];V)$ denotes the set of geometric $\alpha$-Hölder rough paths $\mathbf{x}$ such that the first component $x^1$ of the underlying path $x$ is equal to running time, see Definition~\ref{def:alpha-Holder-RP}.
As noted in the smooth case in Section~\ref{sec:sig-smooth-path}, the signature $\sigx^{<\infty}_{0,T}$ of a time-extended rough path $\mathbf{x} \in \widehat{\mathscr{C}}^\alpha_g([0,T];V)$ characterizes $\mathbf{x}$ -- and, hence, $x$ up to the initial value $x_0$.
This, together with the fact that linear functionals of the signature form an algebra -- also already mentioned for smooth paths in Section~\ref{sec:sig-smooth-path}, but equally true for rough paths -- allow us to apply the Stone-Weierstrass theorem.
\begin{theorem}\label{thr:universality}
  For any compact subset $\mathcal{K} \subset \widehat{\mathscr{C}}^\alpha_g([0,T];V)$, any continuous function $f: \widehat{\mathscr{C}}^\alpha_g([0,T];V) \to \R$, and any $\epsilon > 0$, we can find a linear functional $\ell \in \mathcal{W}_d$ such that
  \begin{equation*}
    \sup_{\mathbf x \in \mathcal{K}} \abs{f(\mathbf x) - \ip{\ell}{\sigx^{<\infty}_{0,T}}} < \epsilon.
  \end{equation*}
\end{theorem}


\begin{proof}[Proof of Theorem~\ref{thr:universality}.]
  For completeness, we give a short proof of this well-known result, a direct consequence of the Stone-Weierstrass theorem.
  Hence, we need to prove that
  \begin{equation*}
    \mathcal{A} \coloneqq \Set{\mathbf{x} \mapsto \ip{\ell}{\sigx^{<\infty}_{0,T}} | \ell \in \mathcal{W}_d, \ \mathbf{x} \in \widehat{\mathscr{C}}^\alpha_g([0,T];V)}
  \end{equation*}
  is a subalgebra of $C\left( \widehat{\mathscr{C}}^\alpha_g([0,T];V);\R \right)$ which is point-separating and contains a non-zero constant function.
  First note that for any $\ell \in \mathcal{W}_d$, the map $\mathbf{x} \mapsto \ip{\ell}{\sigx^{< \infty}_{0,T}}$ is continuous, see, for instance, \cite{friz2010multidimensional}.
  $\mathcal{A}$ is point-separating since the signature $\sigx^{<\infty}$ uniquely determines the rough path $\mathbf{x}$ by the discussion above.
  The constant function $\mathbf{x} \mapsto 1$ is obviously contained in $\mathcal{A}$ by simply choosing $\ell = \emptyset$, where $\emptyset$ denotes the empty word.

  This leaves us to prove that $\mathcal{A}$ is an algebra.
  In addition to the concatenation product, $\mathcal{W}_d$ is also equipped with a commutative shuffle product $\shuffle: \mathcal{W}_d \times \mathcal{W}_d \to \mathcal{W}_d$, see, e.g., \cite{bayer2021optimal}, and the following \emph{shuffle identity} holds for signatures:
  For $\ell_1, \ell_2 \in \mathcal{W}_d$ and a geometric rough path $\mathbf{x}$, we have
  \begin{equation}
    \label{eq:shuffle_identity}
    \ip{\ell_1}{\sigx_{s,t}^{<\infty}} \ip{\ell_2}{\sigx_{s,t}^{<\infty}} = \ip{\ell_1 \shuffle \ell_2}{\sigx_{s,t}^{<\infty}}.
  \end{equation}
  Hence, $\mathcal{A}$ is an algebra and the proof is complete.
\end{proof}
\begin{example}\label{ex_cont_func}
Let $g$ be the solution map of \eqref{rde}, i.e., $y=g(\mathbf x)$, mapping a rough path $\mathbf{x}$ to a path $y$. Then, we  know by Theorem \ref{thm_ex_un_con} that this map is continuous. For that reason, a potential continuous functional of interest in Theorem \ref{thr:universality} can be the $i$th component $y_i$ of $y$ evaluated at $T$ (e.g., think of $W=\mathbb R^d$) meaning that $f(\mathbf x)= y_i(T)$.
\end{example}

As noted in Section~\ref{sec:sign-semi-mart}, we can also lift semimartingales $X$ to the associated rough path $\mathbf{X}$ taking values in $\widehat{\mathscr{C}}^\alpha_g([0,T];V)$-- assuming $X^1_t \equiv t$ --, and Theorem~\ref{thr:universality} applies:
\begin{corollary}
  \label{cor:universality-semimartingales}
  Given a continuous functional $F: \widehat{\mathscr{C}}^\alpha_g([0,T];V) \to \R$ and a compact subset $\mathcal{K} \subset \widehat{\mathscr{C}}^\alpha_g([0,T];V)$. Then, for every $\epsilon > 0$, there is $\ell \in \mathcal{W}_d$, such that
  \[
    \sup_{\omega \in \mathbf{X}^{-1}(\mathcal{K})} \abs{F(\mathbf{X}(\omega)) - \ip{\ell}{\sigX^{<\infty}_{0,T}(\omega)}}<\epsilon,
  \]
  interpreting the rough path lift as a random variable $\mathbf{X}: \Omega \to \widehat{\mathscr{C}}^\alpha_g([0,T];V)$.
\end{corollary}
We also refer to \cite[Corollary 3.8]{cuchiero2022universal} for a more general version applicable to jump-semimartingales.


\subsection{Dynamic approximations with signatures}
\label{sec:dynamic-signature}

Theorem~\ref{thr:universality} shows that a \christian{continuous} function $f(x|_{[0,T]})$ of a path\footnote{For simplicity, we will consider a smooth path $x$ for the time being, but the discussion is equally valid for the rough case.} defined on $[0,T]$ can be approximated by linear functionals of the path's signature $\ip{\ell}{\mathbbm{x}^{<\infty}_{0,T}}$ on the interval $[0,T]$.
Suppose that we are instead given a path-valued functional, i.e., $y(t) = f(t, x|_{[0,t]})$ -- think of a stochastic process $y$ adapted to the filtration generated by $x$.
Under the required regularity conditions, Theorem~\ref{thr:universality} immediately implies that we can find approximations
\[
  y(t) = f(t, x|_{[0,t]}) \approx \ip{\ell(t)}{\mathbbm{x}^{<\infty}_{0,t}}
\]
for suitable linear functionals $\ell(t)$. 
Can we find \emph{uniform approximations} in the sense that
\[
  y(t) = f(t, x|_{[0,t]}) \approx \ip{\ell}{\mathbbm{x}^{<\infty}_{0,t}}
\]
with a linear functional $\ell$ independent of $t$?
On a computational side, this seems plausible, since $\mathbbm{x}^{<\infty}_{0,t}$ contains all monomials in $t$ when $x^1_t \equiv t$.

To formalise the problem, we introduce the space of \emph{stopped rough paths}, going back to \cite{dupire2019functional},  see also, \cite{KLPA20,bayer2021optimal}, for a more didactic presentation we also refer to \cite{bayer_hager_riedel_book24}.

Indeed, consider $\mathcal{C}_T \coloneqq \bigcup_{t \in [0,T]} C([0,t]; V)$ understood as a disjoint union.
Note that an element $x|_{[0,t]} \in \mathcal{C}_T$ is the restriction of a continuous path defined on the interval $[0,T]$ to a sub-interval $[0,t]$, $t\le T$.
A metric is defined on $\mathcal{C}_T$ setting
\[
  d_{\mathcal{C}_T}(x|_{[0,t]}, y|_{[0,s]}) \coloneqq \sup_{u \in [0,T]} \abs{ x_{t \wedge u} - y_{s \wedge u}} + \abs{t-s},
\]
and it turns out that $\mathcal{C}_T$ equipped with this metric is a Polish space.

We will use this framework to study dynamic functional approximations, i.e., we want so approximate time dependent functions of a path of the form $f(t, x|_{[0,t]})$, $0 \le t \le T$.
Note that we can easily understand such a function $f$ as a function $F: \mathcal{C}_T \to \R$.
The natural question is now whether such functions of restrictions of paths can also be approximated by linear functionals of the signature.

Of course, the space $\mathcal{C}_T$ is too large to allow for such approximations.
As before, we will restrict ourselves to a proper rough path version thereof.
Let, for fixed $\alpha$, denote
\begin{equation}
  \label{eq:stopped-rough-paths}
  \Lambda_T^\alpha \coloneqq \bigcup_{t \in [0,T]} \widehat{\mathscr{C}}^\alpha_g([0,t]; V).
\end{equation}
(We restrict ourselves to the time-extended case from the beginning because any such universal approximation result will anyway require us to add time as a component.) 
Using a similar definition (see \cite{bayer2023primal}), we can define a rough path metric $d$ on $\Lambda^\alpha_T$, under which it is a Polish space.
In this framework, repeating the arguments used in the proof of Theorem~\ref{thr:universality} gives (see \cite{bayer_hager_riedel_book24} for more details)
\begin{corollary}
  \label{cor:dynamic-sig-universality}
  Let $\mathcal{K} \subset \Lambda^\alpha_T$ be compact and $F \in C\left( \Lambda^\alpha_T; \R \right)$, then for every $\epsilon > 0$ there is an $\ell \in \mathcal{W}_d$ such that
  \[
    \sup_{\mathbf{x}|_{[0,t]} \in \mathcal{K}} \abs{F(\mathbf{x}|_{[0,t]}) - \ip{\ell}{\sigx^{<\infty}_{0,t}}} < \epsilon.
  \]
  The analogue result for lifts $\mathbf{X}$ of semimartingales also holds a.s.
\end{corollary}

\begin{remark}
  \label{rem:Lp-universality}
  The above universal approximation theorems -- as typical in the area -- are formulated in terms of uniform convergence on compacts.
  We admit that this concept is problematic, especially in infinite dimensions, as compact sets are often very small.
  Note that when we consider stochastic processes, e.g., the lift $\mathbf{X}|_{[0,t]}$ of a semimartingale taking values in $\Lambda^\alpha_T$ (a Polish space), \emph{tightness} implies that for every $\delta > 0$ we can find a compact subset $\mathcal[K] \subset \Lambda^\alpha_T$ such that $P\left( \mathbf{X} \notin \mathcal{K} \right) < \delta$, which allows us to replace ``$\epsilon$-close on compact sets'' by ``$\epsilon$-close with probability $1-\delta$''.

  More general global universal approximation results can be formulated in terms of \emph{weighted spaces} (see \cite{cuchiero2023global}) or in terms of $L^p$-norms by using so-called ``robust signatures'' (see \cite{chevyrev2018signature} for the definition of robust signatures and \cite{bayer2023primal, SA23} for the universality in $L^p$.)
\end{remark}

\subsection{Example: Signature models in finance}
\label{sec:exampl-sign-models}


As motivated by the universality of the signature, we consider the problem of approximating a fixed linear functional $\ell$ of the signature $\sigx^{\le m}_{0,T}$ of a time-augmented smooth path, or, alternatively, a fixed linear functional $\ell$ of the signature $\sigX^{\le m}_{0,T}$ of a (time-extended) continuous semimartingale, by a linear functional of an alternative path $\widetilde {\sigx}$ or $\widetilde {\sigX}$ taking values in a space $\R^{\widetilde n}$ with $\widetilde n \ll \dim T^m(V)$.

As guiding example, we consider \emph{signature models} for financial markets as introduced in \cite{cuchiero2022signature}.
Consider a $d$-dimensional underlying base process $X$, which is assumed to be a (time-extended) continuous semimartingale.
Given a linear functional $\ell \in \mathcal{W}_d$, we then consider an asset price process given by
\begin{equation}
  \label{eq:signature-model}
  S_t = S^{(\ell)}_t \coloneqq \ip{\ell}{\sigX^{\le m}_{0,t}},
\end{equation}
where the truncation degree $m$ is chosen high enough that all words in $\ell$ of degree larger than $m$ have zero coefficient.
Note that conditions on $\ell$ can be formulated such that the resulting price process satisfies fundamental requirements of mathematical finance such as no-arbitrage.\footnote{As such conditions are usually easier to formulate w.r.t.~It\^{o} rather than Stratonovich form, those conditions may look simpler in terms of the associated It\^{o} signature.}
Universality of signatures implies that signature models form a very flexible class of models, capable of approximating many desirable properties of asset price models.
In addition, even though the price process~\eqref{eq:signature-model} is not a Markov process, efficient numerical routines for option pricing exist in some cases.\footnote{Essentially, when the payoff function can be efficiently approximated by a linear functional of the signature of $S$.}
As explained in \cite{cuchiero2022signature}, the model parameter $\ell$ can be efficiently estimated either from time series of asset prices, or from option price data.
However, the dimension of the ``state space'' $T^m(V)$ of~\eqref{eq:signature-model} is generally very large.
Indeed, the dimension (in this sense) of the numerical examples in \cite{cuchiero2022signature} are generally of order $100$, but more complex markets might easily lead to much higher dimensional approximations.

In this paper, we assume that we are given a calibrated signature model $S = S^{(\ell)}$ for a fixed linear functional $\ell$ of degree $m$.
The purpose is to derive reduced order models, i.e., a linear SDE with solution $\widetilde \sigX$ in $\R^{\widetilde n}$ and a linear functional $\widetilde{\ell}: \R^{\widetilde n} \to \R$ such that
\begin{align}\label{form_aim}
\ip{\widetilde{\ell}}{\widetilde \sigX} \approx S^{(\ell)}\quad \text{and} \quad \widetilde n \ll \dim T^m(V).
\end{align}
Note that the requirement of $\widetilde \sigX$ to solve a linear SDE is natural, given that already $\sigX^{\le m}$ solves a linear SDE, see \eqref{eq:signature-ODE-truncated}.


\printbibliography

\end{document}